\documentclass[reqno, 12pt]{amsart}

\usepackage[]{_pascal/_pascal}

\hypersetup{
    pdftitle={},
    pdfauthor={J.~Pascal Gollin, Martin Milanič, Laura Ogrin},
    debug=true,
}

\newtheorem*{condition*}{Condition}
\newtheorem{condition}[thm]{Condition}
\zcRefTypeSetup{condition}{
    Name-sg = Condition,
    name-sg = condition,
    Name-pl = Conditions,
    name-pl = conditions,
}

\newcommand{\Gc}{\overline{G}}
\DeclareMathOperator{\diam}{diam}
\newcommand{\cP}{\mathcal{P}}
\newcommand{\R}{\mathbb{R}}

\usepackage[normalem]{ulem}

\newcommand{\turan}[2]{\ensuremath{T_{ #1, #2 }}}
\newcommand{\doublestar}[2]{\ensuremath{S_{#1, #2}}}
\newcommand{\triplestar}[3]{\ensuremath{S_{#1, #2, #3}}}

\begin{document}

\author[Gollin]{J.~Pascal Gollin$^\dagger$}
\address[$\dagger$]{FAMNIT, University of Primorska, Koper, Slovenia.}
\email{\tt pascal.gollin@famnit.upr.si}

\author[Milani\v{c}]{Martin Milani\v{c}$^{\dagger,\ddagger}$}
\address[$\ddagger$]{IAM, University of Primorska, Koper, Slovenia.}
\email{\tt martin.milanic@upr.si}

\author[Ogrin]{Laura Ogrin$^{\dagger}$}
\email{\tt laura.ogrin@famnit.upr.si}

\title[]{Minimal toughness in subclasses of weakly chordal graphs}

\date{\today}



\begin{abstract}
The \emph{toughness} of a graph $G$ is defined as the largest real number $t$ such that for any set $S\subseteq V(G)$ such that $G-S$ is disconnected, $S$ has at least $t$ times more elements than $G-S$ has components (unless $G$ is complete, in which case the toughness is defined to be infinite).
A graph is said to be \emph{minimally tough} if deleting any edge decreases the toughness.
It is an open question whether there exists a minimally tough non-complete chordal graph with toughness exceeding~$1$.
We initiate the study of minimally tough graphs in the larger class of weakly chordal graphs.
We obtain complete classifications of minimally tough graphs in the following subclasses of weakly chordal graphs: co-chordal graphs whose complement has diameter at least~$3$, net-free co-chordal graphs, complements of forests, $P_4$-free graphs, and complete multipartite graphs.
Our approach leads to simple proofs of two results on minimally tough graphs due to Dallard, Fernández, Katona, Milanič, and Varga.
\end{abstract}

\maketitle

\section{Introduction}
\label{sec:intro}

Toughness is a measure of connectedness of graphs that, roughly speaking, gives a lower bound on the ratio between the cardinality of any separator in the graph and the number of components of the remaining graph.
More precisely, for a real number $t$, a graph $G$ is said to be \emph{$t$-tough} if ${\abs{S} \geq t \cdot c(G-S)}$ holds for every set ${S \subseteq V(G)}$ such that~${G-S}$ is disconnected, where~${c(H)}$ denotes the number of connected components of a graph~$H$. 
Every Hamiltonian graph is $1$-tough, and Chv\'atal conjectured that there exists a real number~$t_0$ such that every $t_0$-tough graph is Hamiltonian (see~\cite{Chva73,Bauer2006,Broersma2015}).
This conjecture is still open.

The \emph{toughness} of a graph~$G$, denoted~$\tau(G)$, is defined as the maximum real number~$t$ such that~$G$ is $t$-tough (unless~$G$ is complete, in which case the toughness is defined to be~$\infty$). 
Since deleting an edge from a graph $G$ cannot increase the toughness, every graph $G$ with toughness $t$ contains a minimal spanning subgraph with toughness~$t$.
Such a subgraph is necessarily \emph{minimally $t$-tough}, that is, deleting any edge from the graph results in a graph with toughness strictly smaller than $t$.
Consequently, Chv\'atal's conjecture is equivalent to the existence of a real number $t_0$ such that every minimally $t_0$-tough graph is Hamiltonian.
This motivates the study of graphs that are minimally $t$-tough for some $t$, or, shortly, \emph{minimally tough} graphs.

This problem is far from trivial.
For every positive rational number $t$, every graph is an induced subgraph of a minimally $t$-tough graph (see~\cite{Katona2018}) and the problem of recognizing minimally $t$-tough graphs is \textsf{DP}-complete\footnote{The complexity class \textsf{DP}, introduced by Papadimitriou and Yannakakis in 1984 (see~\cite{Papa1984}) is the class of all languages that can be expressed as the intersection of a language in \textsf{NP} and a language in \textsf{coNP}.} (see~\cite{Kato21}). 
For every positive rational number $t\le 1$, the complexity of recognizing $t$-tough bipartite graphs is \textsf{coNP}-complete (see~\cite{Katona2023}); however, to the best of our knowledge, the complexity of recognizing minimally $t$-tough bipartite graphs is open.

Minimally $t$-tough graphs were characterized for classes of claw-free graphs for $t < 3/2$ (see~\cite{Katona2018,Kato23}), split graphs for all $t$ (see~\cite{Kato23}), chordal graphs for $t\le 1$ (see~\cite{Kato23,Kato24}), interval graphs for all $t$ (see~\cite{Kato24,Dall23}), and strongly chordal graphs for all $t$, as a combination of results from~\cite{Kato23,Kato24,Dall23} (see also~\cite[Corollary 3.8]{ogrin2025characterizations}). 
Katona and Varga~\cite{Kato23} showed that the following graph classes can be recognized in polynomial time: minimally $t$-tough, split graphs for any positive rational number $t$, minimally $t$-tough, $2K_2$-free graphs for any positive rational number $t$, and minimally $t$-tough, claw-free graphs for any positive rational number $t\le 1$.
Furthermore, the fact that toughness of graphs with bounded treewidth can be computed in polynomial time (see~\cite{KK2024}) leads to a polynomial-time algorithm for recognizing minimally $t$-tough graphs with bounded treewidth, for any positive rational number $t$.
For the class of $2K_2$-free graphs, further results on minimally $t$-tough graphs were obtained by Ma, Hu, and Yang (see~\cite{MR4674928,MR4927512}).

Dallard et al.~\cite{Dall23} conjectured that for any real number $t> 1/2$, there exists no minimally $t$-tough, chordal graph.
In fact, while it is known that there is no minimally $t$-tough chordal graph with $1/2<t\le 1$ (see~\cite{Kato23}), it is not known if there exists a real number~$\hat t$ such that each minimally $t$-tough chordal graph with finite toughness satisfies $t\le \hat t$.

Kriesell conjectured that every minimally $1$-tough graph contains a vertex of degree~$2$ (see~\cite{Krie03}).
A more general conjecture, called the generalized Kriesell's conjecture, proposed by Katona and Varga (see~\cite{Kato23}), states that for every $t>0$, every minimally $t$-tough graph has a vertex of degree~$\ceil{2t}$.
The generalized Kriesell's conjecture was disproved by Zheng and Sun~\cite{Zhen24}, even for the class of $2K_2$-free graphs.
It has also been disproved for line graphs (and hence, for claw-free graphs), by Hasanvand~\cite{hasanvand2025existenceminimallytoughgraphs}.
On the other hand, the conjecture has been verified for classes of split graphs~\cite{Kato23}, chordal graphs for $t\le 1$ (see~\cite{Kato23,Kato24}), and claw-free graphs for $t < 2$ (see~\cite{Katona2018,Kato23, Ma2023}; see also~\cite{Ma2025}).
Kriesell's conjecture is still open, but has been verified, in addition to chordal graphs and claw-free graphs, for $2K_2$-free graphs (see~\cite[Theorem 1.7]{MR4674928}), for graphs with independence number at most~$3$~\cite{MR4866588}, and for some other cases~\cite{Cao2026}.
See~\cite[Section 3.2]{ogrin2025characterizations} for more details.\\

The starting point of our research is the aforementioned conjecture of Dallard et al.~\cite{Dall23} on minimally tough, chordal graphs.
We explore minimal toughness in a more general class of graphs, the class of \emph{weakly chordal graphs}, which are defined as graphs that exclude cycles of length at least $5$ and their complements as induced subgraphs (see, e.g.,~\cite{MR1991712,MR815392,MR2335297}).
The class of weakly chordal graphs is closed under complementation and contains, besides all chordal graphs, also their complements (known as \emph{co-chordal} graphs), as well as all $P_4$-free graphs, that is, graphs not containing the four-vertex path as an induced subgraph. 

While, as mentioned before, no minimally tough chordal graphs with (finite) toughness exceeding $1/2$ are known, we show that there exist minimally tough weakly chordal graphs with arbitrarily large finite toughness.
This is a consequence of our main results, complete classifications of minimally tough graphs in each of the following subclasses of weakly chordal graphs:
\begin{itemize}
    \item co-chordal graphs whose complements have diameter different from~$2$,
    \item net-free co-chordal graphs, 
    \item complements of forests,
    \item $P_4$-free graphs, and
    \item complete multipartite graphs.
\end{itemize}
In particular, we show that there exist minimally tough complete multipartite graphs with arbitrarily large toughness.

To state our results in detail, we need to introduce some definitions. 
We refer to complete graphs and edgeless graphs as \emph{trivially} minimally tough and to all other minimally tough graphs as \emph{non-trivially} minimally tough.
A graph is \emph{complete multipartite} if it admits a partition of its vertex set into subsets called \emph{parts} such that two distinct vertices of the graph are adjacent if and only if they belong to different parts.
For a sequence of positive integers ${n_1 \leq n_2 \leq \ldots \leq n_k}$, we denote by~${K_{n_1,\ldots,n_k}}$ the complete multipartite graph with~$k$ parts ${V_1,\ldots, V_k}$ such that ${\abs{V_i} = n_i}$ for all~${i \in [k]}$.
We denote by~$\turan{n}{k}$ the complete multipartite graph with~$n$ vertices and~$k$ parts, whose parts are as equal in cardinality as possible.\footnote{These graphs are often referred to as \emph{Tur\'{a}n graphs}.} 
We will mostly be interested in the case when all parts have cardinality~$2$, except for possibly one, which has cardinality~$1$, that is, in the graphs $\turan{2\ell}{\ell}$ and $\turan{2\ell-1}{\ell}$, where $\ell\ge 2$.

\begin{theorem}
    \label{P4free:mintough}
    A $P_4$-free graph is non-trivially minimally tough if and only if it is isomorphic to one of~$K_{2,3}$, $K_{1,\ell}$, $\turan{2\ell}{\ell}$, or $\turan{2\ell-1}{\ell}$, for some~${\ell \geq 2}$. 
\end{theorem}

Since all complete multipartite graphs are $P_4$-free and all minimally tough $P_4$-free graphs are complete multipartite, \zcref{P4free:mintough} implies the following classification of minimally tough complete multipartite graphs.

\begin{theorem}
    \label{complete-char-multipar}
    A complete multipartite graph is non-trivially minimally tough if and only if it is isomorphic to one of~$K_{2,3}$, $K_{1,\ell}$, $\turan{2\ell}{\ell}$, or $\turan{2\ell-1}{\ell}$ for some~${\ell \geq 2}$. 
\end{theorem}

Given a graph $G$, the \emph{co-diameter} of $G$ is the diameter of the complement of $G$. 

\begin{theorem}[store=cochordal]
\label{characterization-co-chordal}
     A co-chordal graph with co-diameter at least~$3$ is non-trivially minimally tough if and only if it is isomorphic to one of  $P_4$, $K_{2,3}$, $K_{1,\ell}$, $\doublestar{k}{\ell}$, $\turan{2\ell}{\ell}$, or $\turan{2\ell-1}{\ell}$ for some ${k \geq 1}$ and~${\ell \geq 2}$. 
 \end{theorem}

We obtain a similar classification for net-free co-chordal graphs.
The \emph{net} is the graph obtained from the complete graph $K_3$ by attaching a pendant edge to each vertex.

\begin{theorem}[store=netfreecochordal]
\label{characterization-net-free-co-chordal}
     A net-free co-chordal graph is non-trivially minimally tough if and only if it is isomorphic to one of $P_4$, $K_{2,3}$, $K_{1,\ell}$, $\doublestar{k}{\ell}$, $\turan{2\ell}{\ell}$, or $\turan{2\ell-1}{\ell}$ for some ${k \geq 1}$ and~${\ell \geq 2}$. 
 \end{theorem}

It is worth noting that the class of net-free co-chordal graphs contains all complements of interval graphs and, more generally, all complements of strongly chordal graphs; see~\cite{BrandstaedtLS:1999:GraphClasses}.

 \begin{theorem}[store=coforests]
     \label{characterization-co-forests}
     A complement of a forest is non-trivially minimally tough if and only if it is isomorphic to one of $P_4$, $\turan{2\ell}{\ell}$, or $\turan{2\ell-1}{\ell}$ for some~${\ell \geq 2}$.
\end{theorem}

For a quick reference, we collect in \zcref{table:toughness} the values of the toughness of all the minimally tough graphs appearing in the characterizations from \zcref{P4free:mintough,characterization-co-chordal,characterization-net-free-co-chordal,characterization-co-forests,complete-char-multipar}. 
The entries in the table follow from~\zcref{prop3} and the fact that the toughness of a tree $T$ with maximum degree $\Delta$ is equal to $1/\Delta$ (see, e.g.~\cite{Broersma2015}).

    {\renewcommand{\arraystretch}{1.2}
    \begin{table}[h!]
        \begin{tabular}{|c||c|c|c|c|c|c|}
        \hline
            Graph & 
            $P_4$ & 
            $\doublestar{k}{\ell}$ &
            $K_{2,3}$ & 
            $K_{1,\ell}$ & 
            $\turan{2\ell}{\ell}$ & 
            $T_{2\ell-1,\ell}$ 
            \\ \hline
             Toughness & 
            $\frac{1}{2}$ & 
            $\frac{1}{\ell+1}$ &
            $\frac{2}{3}$ & 
            $\frac{1}{\ell}$ & 
            $\ell-1$ & 
            $\ell-\frac{3}{2}$ \\
            \hline
        \end{tabular}
        \caption{Toughness of the various minimally tough graphs we are considering, where $\ell \geq 2$ and~${1 \leq k \leq \ell}$.}
        \label{table:toughness}
    \end{table}}

It is easy to verify that all the graphs that appear in any of our characterizations contain a vertex of degree $\lceil 2t \rceil$, where~$t$ is the respective toughness of the graph. 
As a consequence, the generalized Kriesell's conjecture holds for the classes of $P_4$-free graphs, co-chordal graphs with co-diameter at least $3$, net-free co-chordal graphs, and complements of forests. 

\bigskip
\paragraph{Our approach}
In all of the above results, our approach is to first consider the case when the complement of the graph is disconnected, that is, when the graph is the join of two smaller graphs.
In this regard, we develop the following necessary condition for the join of two graphs to be minimally tough.

\begin{condition}[The join condition]
    \label{joincondition}

If one of the two graphs contains a vertex of maximum degree in the join, then the other graph must be regular.
\end{condition}
This condition is instrumental in the proof of our classification of minimally tough $P_4$-free graphs.
Furthermore, it leads to simple proofs of two known results on minimally tough graphs due to Dallard et al.~\cite{Dall23}:
\begin{itemize}
    \item The fact that for $t > 1$, no chordal graph with a universal vertex is minimally $t$-tough. 
\item A classification of minimally $t$-tough graphs with a universal vertex for all $t\le 1$.
In fact, we extend this classification to all $t\le 3/2$.
\end{itemize}
Our approach for both $P_4$-free graphs and complements of connected chordal graphs with diameter at least $3$ is also based on the existence of a \emph{dominating edge} in the graph, that is, an edge such that every vertex of the graph is adjacent to at least one of the endpoints of the edge.
Applying a lemma from Dallard et al.~\cite{Dall23}, we characterize the cases when deleting some dominating edge does not change the toughness of the graph, implying that the graph is not minimally tough.
As another step towards the classification of minimally tough $P_4$-free graphs, we determine the toughness of complete multipartite graphs (see \zcref{prop3}), generalizing the analogous result of Chv\'{a}tal~\cite{Chva73} for complete bipartite graphs.

Lastly, the classification of minimally tough graphs within the class of net-free co-chordal graphs is obtained by reducing the problem to complements of chordal graphs with diameter at least $3$, using a characterization of hereditary closed neighbourhood-Helly graphs due to Groshaus and Szwarcfiter~\cite{MR2383735}.

\bigskip
\paragraph{Structure of the paper}
In \zcref{sec:prelim} we summarize some of the standard graph theory terminology and notation, as well as some folklore results on connectivity and chordal graphs.
In \zcref{sec:bip-matchings,sec:prelim-tough}, we develop some basic tools that we apply in later sections.
More specifically, in \zcref{sec:bip-matchings}, we make use of matchings in bipartite graphs to determine the local connectivity in certain cases.
In \zcref{sec:prelim-tough} we present some preliminary results on toughness and minimal toughness of graphs.
In particular, focusing on graphs with dominating edges, we determine the toughness of complete multipartite graphs.
In \zcref{sec:joins} we explore minimal toughness of joins of graphs and apply our findings to graphs with a universal vertex.
In \zcref{sec:P4free}, we characterize minimally tough $P_4$-free graphs.
In \zcref{sec:cochordal}, we obtain our most involved result, a characterization of minimally tough co-chordal graphs with co-diameter at least $3$.
As a consequence, we also classify minimally tough net-free co-chordal graphs and minimally tough complements of forests.
We conclude the paper with some open questions in \zcref{sec:conclusion}.

\section{Preliminaries}
\label{sec:prelim}

For a positive integer $n$, we denote by $[n]$ the set of first $n$ positive integers.
For any graph notions and notation not explained here, we refer to~\cite{Dies05}.
All the graphs considered in this paper are finite, simple, and undirected. 
The graph of order $0$ is called the \emph{null graph}, and a graph of order at most $1$ is called \emph{trivial}.
For a vertex $v$ in a graph $G$, the degree, the open neighbourhood, and the closed neighbourhood of $v$ in $G$ are denoted by $d_G(v)$, $N_G(v)$, and $N_G[v]$, respectively, with the subscript $G$ omitted when the graph is clear from the context. 
The minimum and maximum degree of a graph $G$ are denoted by $\delta(G)$ and $\Delta(G)$, respectively.

For a graph $G$ and an arbitrary set $U \subseteq V(G)$, we denote by $G[U]$ the induced subgraph of $G$ with vertex set $U$.
For any set $U \subseteq V(G)$, we denote by $G-U$ the graph obtained from $G$ by deleting all the vertices from $U$ and all the edges incident with at least one vertex from $U$, that is, $G-U=G[V(G)\setminus U]$. 
If $U=\{u\}$, we write $G-u$ rather than $G-\{u\}$.
For an edge $e=uv \in E(G)$, we denote by $G-e$ or $G-uv$ the graph obtained from $G$ by removing the edge $e$.
For disjoint subsets $U,W \subseteq V(G)$, an edge $uw$ with $u \in U$ and $w \in W$ is said to be a \emph{$U$-$W$ edge}; furthermore, we let $G[U,W]$ denote the graph whose vertex set is $U \cup W$ and whose edge set consists of all $U$-$W$ edges.

The \emph{complement} $\overline{G}$ of a graph $G$ is the graph with vertex set $V(G)$ defined by $uv \in E(\Gc)$ if and only if $uv \notin E(G)$.
Let $G_1=(V_1,E_1)$ and $G_2=(V_2,E_2)$ be two vertex-disjoint graphs.
The \emph{disjoint union} of $G_1$ and $G_2$, is the graph $G_1+G_2$ whose vertex set is $V_1 \cup V_2$ and the edge set is $E_1 \cup E_2$. 
Furthermore, for $k\geq 1$ and a graph $G$, we denote by $kG$ the graph consisting of $k$ pairwise disjoint copies of $G$.
 The \emph{join} of $G_1$ and $G_2$, denoted by $G_1 * G_2$, is the graph whose vertex set is $V_1 \cup V_2$ and  the edge set is the union of $E_1$, $E_2$ and all the edges joining $V_1$ and $V_2$.

For a positive integer $n$, we denote by $P_n$, $C_n$, and $K_n$, respectively, the $n$-vertex path graph, the $n$-vertex cycle graph (for $n\ge 3$), and the $n$-vertex complete graph.
For a graph $G$ and two vertices $u,v \in V(G)$, a \emph{$u$-$v$ path in $G$} is a subgraph in $G$ that is a path from $u$ to $v$. 
Two $u$-$v$ paths are \emph{internally vertex-disjoint} if $u$ and $v$ are their only common vertices.
For two sets of vertices $U,W \subseteq V(G)$, we call $P=(v_1,\ldots, v_k)$ a $U$-$W$ path if $V(P) \cap U=\{v_1\}$ and $V(P) \cap W =\{v_k\}$.

The \emph{distance} between two vertices $u$ and $v$ in a graph $G$, denoted by $d_G(u,v)$, is the  length of a shortest $u$-$v$ path in $G$. If no such path exists, we set $d_G(u,v):=\infty$.
The \emph{diameter} of $G$, denoted by $\diam(G)$, is the largest distance between any two vertices in $G$, i.e., $\diam(G)=\max_{u,v \in V(G)}d_G(u,v)$.
The \emph{eccentricity} of a vertex $u \in V(G)$, denoted by $\epsilon(u)$, is $\max_{v \in V(G)}d_G(u,v)$. Hence, $\diam(G) = \max_{u \in V(G)} \epsilon(u)$.

We denote by $c(G)$ the number of components of a graph $G$. 
A subset $S \subseteq V(G)$ is a \emph{separator} of $G$ if the graph $G-S$ has more than one component, i.e., if $c(G-S)\geq 2$. In this case we say that $S$ \emph{separates} $G$. 
For two vertices $u,v \in V(G)$, a set $S \subseteq V(G)\setminus \{u,v\}$ is a \emph{$u$-$v$ separator} if $G-S$ has no $u$-$v$ path. In this case we say that $S$ separates $u$ and $v$ in $G$.
For an integer $k\geq 0$, a graph $G$ is said to be \emph{$k$-connected} if ${\abs{V(G)} \geq k+1}$ and $G-S$ is connected for every set ${S \subseteq V(G)}$ with ${\abs{S} \leq k-1}$. 
In other words, no two vertices of $G$ are separated by fewer than $k$ other vertices.
The largest integer $k$ such that $G$ is $k$-connected is the \emph{connectivity} of $G$, denoted by $\kappa(G)$.
For arbitrary distinct vertices $u, v \in V(G)$, we denote by $\kappa_G(u,v)$ the corresponding \emph{local connectivity}, that is, the maximum number of pairwise internally vertex-disjoint $u$-$v$ paths in $G$. 
If $u$ and $v$ are adjacent, we also count the path $(u,v)$.

Next we recall two related inequalities.
The first one follows from the ``Global version of Menger's Theorem''~{\cite[Theorem~3.3.6]{Dies05}} (see also~{\cite[Theorem~4.2.21]{West00}}), while the second one follows directly from the definitions.

\begin{lemma}\label{conn-ivdp} 
Let $u$ and $v$ be two distinct vertices in a graph $G$.
Then 
\[\kappa(G) \leq \kappa_G(u,v) \leq \min\{d(u),d(v)\}\,.\]
\end{lemma}

The graph isomorphism relation is denoted by $\cong$.
A \emph{graph class} is a collection of graphs that is closed under isomorphism. 
A graph class is \emph{hereditary} if it is closed under induced subgraphs. 

A graph $G$ is \emph{$P_4$-free} if it has no induced subgraph isomorphic to $P_4$.
A \emph{star} is a graph isomorphic to a complete bipartite graph $K_{1,a}$; the vertex of degree $a$ is the \emph{center} of the star. 
(In the case $a = 1$, both vertices of the star $K_{1,a}$ have degree~$a$; in this case, we choose one of them arbitrarily as the center.)
A \emph{double star} is a graph, denoted by $\doublestar{a}{b}$, obtained from the disjoint union of two stars~$K_{1,a}$ and~$K_{1,b}$ by adding an edge between their centers. 
Similarly, we denote by~$\triplestar{a}{b}{c}$ the graph obtained from the disjoint union of three stars~$K_{1,a}$, $K_{1,b}$, and~$K_{1,c}$ by adding a triangle between their centers.

A graph $G$ is \emph{chordal} if it has no induced cycle of length at least $4$. 
A graph $G$ is \emph{co-chordal} if its complement $\overline{G}$ is chordal.
A vertex of a graph $G$ is \emph{simplicial} if its neighbourhood in $G$ is a clique. 
Every non-trivial chordal graph has two simplicial vertices at maximum distance, which easily follows from a result of Voloshin~\cite{Volo82} and Farber and Jamison~\cite{Farb86} (see also~{\cite[Lemma~5.3.16]{West00}}). 

\begin{lemma}\label{SimplVert-MaxDistance}
Let $G$ be a chordal graph with $\diam(G)=d$ for some positive integer $d$. 
Then there exist simplicial vertices $u$ and $v$ in $G$ at distance $d$. 
\end{lemma}

The existence of simplicial vertices in chordal graphs heavily restricts the structure of regular chordal graphs.

\begin{lemma}\label{regular-chordal}
Let $G$ be a connected chordal graph. If $G$ is $d$-regular for some $d \geq 1$, then $G$ is isomorphic to~$K_{d+1}$.
\end{lemma}
\begin{proof}
Assume $G$ is $d$-regular.
Let $v \in V(G)$ be a simplicial vertex of $G$. Then $N[v] \cong K_{d+1}$. Assume there is a vertex $u \in V(G)\setminus N[v]$. Since $G$ is connected, there exists a $v$-$u$ path $P$ in $G$. Then there  exist a vertex $v' \in N(v)$ and a vertex $u' \notin N[v]$, such that $v'u' \in E(P)$. But then $v'$ is adjacent to $u'$ and all the vertices in $N[v]\setminus \{v'\}$. Hence, $d(v') \geq d+1$, a contradiction. So $V(G)=N[v]$ and hence $G \cong K_{d+1}$.
\end{proof}

\section{Bipartite matchings and local connectivity}
\label{sec:bip-matchings}

We recall some definitions and observations related to matchings in graphs (specifically, in bipartite graphs) and apply them to characterize the structure of certificates for local connectivity of certain vertex pairs in bipartite graphs (\zcref{matching-ivdp}), and derive a formula for the local connectivity of two vertices from different components of the complement of a graph (\zcref{lem:kuv}).
The results of this section will be used in various proofs throughout the paper. 

A \emph{matching} in a graph $G$ is a set of edges with no shared endpoints. 
A \emph{maximum matching} is a matching of maximum size among all matchings in the graph. 

We will consider matchings in a bipartite graph $G$ in relation to collections of internally vertex-disjoint paths in a graph obtained from $G$, as defined below. 
\begin{definition}\label{extension-of-G}
Let $G$ be a bipartite graph with parts $A$ and $B$. 
Let $u$ and $v$ be two new vertices, that is, $\{u,v\}\cap V(G)=\emptyset$. 
The \emph{$(u,v)$-extension} of $(G,A,B)$ is the graph $\widehat{G}$  with $V(\widehat{G})=V(G) \cup \{u,v\}$  
and $E(\widehat{G})=E(G) \cup \{ua \colon a \in A\} \cup \{bv \colon b \in B\}$.
For every matching $M$ in $G$ we denote by $\cP(M)$ the set $\{(u,x,y,v)\colon xy \in M\}$. Notice that $\cP(M)$ is a collection of pairwise internally vertex-disjoint $u$-$v$ paths in $\widehat{G}$.
\end{definition}
The following lemma shows that any maximum collection of pairwise internally vertex-disjoint $u$-$v$ paths in $\widehat{G}$ corresponds to a maximum matching in $G$. 

\begin{lemma}\label{matching-ivdp} 
Let $G$ be a bipartite graph with parts $A$ and $B$. 
Let $\widehat{G}$ be the \hbox{$(u,v)$-extension} of $(G,A,B)$.
Then for every maximum collection of pairwise internally vertex-disjoint $u$-$v$ paths in $\widehat{G}$, there exists a maximum matching $M$ in $G$ such that $\cP(M)=\cP$.
\end{lemma}
\begin{proof}
Let $\cP$ be a maximum collection of pairwise internally vertex-disjoint $u$-$v$ paths in $\widehat{G}$.
Let~${p = \abs{\cP}}$ and let~$m$ be the maximum size of a matching in~$G$. 

First, observe that every path in $\cP$ is of the form $(u,a,b,v)$, where $a \in A$ and $b \in B$. Indeed, since the distance between $u$ and $v$ in $\widehat{G}$ is at least $3$, every  path in $\cP$ has length at least $3$, and any path of length greater than $3$ would be of the form $(u,a_1,b_1,\ldots, a_k,b_k,v)$ for some vertices $a_1,\ldots,a_k \in A$, $b_1,\ldots,b_k \in B$, where $k\geq 2$. Such a path could be replaced by two paths, namely $(u,a_1,b_1,v)$ and $(u,a_k,b_k,v)$, contradicting the maximality of $\cP$.

Second, since the paths in $\cP$ are pairwise internally vertex-disjoint, $\cP$ is of the form $\{(u,a_i,b_i,v)\colon 1 \leq i \leq p\}$ where $a_1,\ldots, a_p \in A$ and $b_1,\ldots, b_p \in B$ are pairwise distinct vertices. Then $M=\{a_ib_i\colon 1 \leq i \leq p\}$ is a matching in $G$ such that $\cP=\cP(M)$.

It remains to show that $M$ is a maximum matching. 
Suppose for a contradiction that this is not the case. 
Then~${\abs{M} = p < m}$. 
Let~$M'$ be a maximum matching in~$G$. Then~${\cP(M')}$ is a collection of pairwise internally vertex-disjoint $u$-$v$ paths in~$\widehat{G}$ with ${\cP(M') = m > p}$, contradicting the maximality of~$\cP$. 
\end{proof}

We now analyze the local connectivity between vertices in joins of graph. 

\begin{lemma}\label{lem:kuv} 
Let $G_1$ and~$G_2$ be two vertex-disjoint graphs and let~${G = G_1 * G_2}$. 
Then for every pair of vertices~${u \in V(G_1)}$ and~${v \in V(G_2)}$, we have
\[
    \kappa_G(u,v) = \min\{d_G(u),d_G(v)\}.
\]
\end{lemma}

\begin{proof}
    For~${i \in [2]}$, let~${V_i \coloneqq V(G_i)}$, 
    let~${n_i \coloneqq \abs{V_i}}$, and let~$d_i$ denote the degree-function of~$G[V_i]$, that is,~$d_{G[V_i]}$.
    
    Let~${u \in V_1}$ and~${v \in V_2}$ be given. 
    By \zcref{conn-ivdp} we have $\kappa_G(u,v) \leq \min\{d_G(u),d_G(v)\}$. 
    We construct a set $\cP$ of internally vertex-disjoint $u$-$v$ paths in $G$ of size ${\min\{d_G(u),d_G(v)\}}$, as follows. 
    
    Let~${\cP_1=\{(u,v)\}}$, let~${\cP_2=\{(u,x,v) \mid x \in N_G(u) \cap N_G(v)\}}$, 
    and let $\cP_3$ be a maximum set of pairwise internally vertex-disjoint $u$-$v$ paths of length~$3$ with whose internal vertices are in~${V(G) \setminus (N_G(u) \cap N_G(v))}$. 
    Clearly, ${\abs{\cP_1} = 1}$ and 
    \[
        \abs{\cP_2} 
        = \abs{N_G(u) \cap N_G(v)} 
        = d_1(u)+d_2(v)\,.
    \]
    To determine the cardinality of~$\cP_3$, 
    let ${A \coloneqq V_2\setminus N_{G_2}[v]}$ and ${B \coloneqq V_1\setminus N_{G_1}[u]}$, and
    notice that every path in~$\cP_3$ is of the form ${(u,a,b,v)}$, where~${a \in A}$ and~${b \in B}$. 
    So $\cP_3$ is a maximum set of pairwise internally vertex-disjoint {$u$-$v$ paths} in the graph ${\widehat{G} \coloneqq G[V_1\setminus N_{G_1}(u),V_2\setminus N_{G_2}(v)]-uv}$.
    Furthermore, since in the graph $\widehat{G}$, vertices $u$ and $v$ are non-adjacent to each other, $u$ is adjacent to every vertex in $A$, and $v$ is adjacent to every vertex in $B$, the graph $\widehat{G}$ is the $(u,v)$-extension of $(G[A,B], A, B)$.
    Hence, by \zcref{matching-ivdp}, $\cP_3=\cP(M)$ for some maximum matching $M$ in  $G[A,B]$. 
    Since $G[A,B]$ is a complete bipartite graph, 
    \[
        \abs{M} 
        = \min\{\abs{A},\abs{B}\} 
        = \min\{n_2-d_2(v)-1,n_1-d_1(u)-1\}
    \] 
    and, consequently, ${\abs{\cP_3} = \min\{n_1-d_1(u),n_2-d_2(v)\} - 1}$. 
 
    Let~${\cP \coloneqq \cP_1 \cup \cP_2 \cup \cP_3}$. 
    Since~$\cP_1$, $\cP_2$, and~$\cP_3$ are pairwise disjoint, we have
    \begin{align*}
        \abs{\cP}
        &= 1+d_1(u)+d_2(v) + \min\{n_1-d_1(u),n_2-d_2(v)\}-1\\
        &= \min\{n_1+d_2(v), n_2+d_1(u)\}\\
        &= \min\{d_G(v),d_G(u)\}.
    \end{align*}
    Since~$\cP$ is a collection of internally vertex-disjoint $u$-$v$ paths, we get
    \[
        {\kappa_G(u,v)\geq \abs{\cP} = \min\{d_G(v),d_G(u)\}}\,,
    \] 
    which completes the proof.
\end{proof}

\section{Preliminary results on toughness and minimal toughness}
\label{sec:prelim-tough}

Let~$G$ be a graph, and ${t \in \R}$. 
Recall that~$G$ is said to be \emph{$t$-tough} if for every ${S \subseteq V(G)}$ that separates~$G$, we have ${t \leq \frac{\abs{S}}{c(G-S)}}$. 
The \emph{toughness} of a non-complete graph~$G$, denoted by~$\tau(G)$, is the largest~$t$ such that~$G$ is $t$-tough.
Since there are no separators in complete graphs, the toughness is defined as~$\infty$ in this case, that is, ${\tau(K_n) = \infty}$ for all~${n \geq 1}$. 
On the other hand, if~$G$ is disconnected, then~${\tau(G) = 0}$. 
The converse also holds, that is, ${\tau(G) = 0}$ if and only if~$G$ is disconnected. 

With the convention that ${\min \emptyset = \infty}$, we can express toughness as
\begin{equation}\label{toughness}
    \tau(G) = \min_{\substack{S\subseteq V(G)\\S \text{ separator}}}\frac{\abs{S}}{c(G-S)}\,.
 \end{equation}
It follows from the definition that for every separator~$S$ of~$G$, we have~${\tau(G) \leq \frac{\abs{S}}{c(G-S)}}$. 
The separators for which equality is attained are called \emph{tough separators}.
Notice that the toughness of any non-complete graph is a rational number. 
Furthermore, by considering a separator~$S$ of minimum size in a non-complete graph~$G$, we immediately obtain the following relation between toughness and connectivity.

\begin{lemma}[Chv\'{a}tal~{\cite[Proposition~1.3]{Chva73}}]\label{lem:toughness-connectivity}
    If~$G$ is a non-complete graph, then ${\tau(G) \leq \frac{\kappa(G)}{2}}$.
\end{lemma}

For~${x \in \R}$ let~${\ceil{x} \coloneqq \min\{m \in \mathbb{Z}\colon m \geq x \}}$, that is, $\ceil{x}$ is the unique integer in the interval~${[x,x+1)}$. 
Since~$\kappa(G)$ is an integer, we get the following corollary.
\begin{corollary}\label{toughness-connectivity}
    If~$G$ is a non-complete graph, then ${\ceil{2\tau(G)} \leq \kappa(G)}$.
\end{corollary}

Observe that by removing an edge in a graph~$G$, the toughness of~$G$ can either stay the same or it decreases. 
In the extreme case, the removal of any edge in~$G$ decreases the toughness of~$G$. 
This motivates the concept of minimally tough graphs.

\begin{definition}
    For~${t \in \mathbb{R}\cup\{\infty\}}$, 
    a graph~$G$ is called \emph{minimally $t$-tough} if~${\tau(G)=t}$ and ${\tau(G-e)<t}$ for all edges~${e \in E(G)}$.
    A graph is called \emph{minimally tough} if it is minimally $t$-tough for some ${t \in \mathbb{R} \cup \{\infty\}}$.
\end{definition}

Observe that if~$G$ is a disconnected graph, then~$G$ is minimally tough, if and only if~$G$ is minimally $0$-tough, if and only if~$G$ is edgeless.
On the other hand, if~$G$ is a complete graph, then ${\tau(G) = \infty}$ and ${\tau(G-e) < \infty}$, for every edge~${e \in E(G)}$, so complete graphs are minimally tough. 
Recall that we refer to complete graphs and edgeless graphs as \emph{trivially} minimally tough.
From now on, we will thus focus on connected non-complete graphs and study \emph{non-trivially} minimally tough graphs, that is, minimally $t$-tough graphs where~$t$ is a positive rational number.
Note that every non-trivially minimally tough graph is connected. 

We will make frequent use of the following characterization of  graphs that are not minimally tough.

\begin{theorem}[Dallard~et al.~{\cite[Theorem~3.2]{Dall23}}]\label{thm31}
    Let~$G$ be a non-complete graph and let~$t$ be the toughness of~$G$. 
    Then~$G$ is not minimally $t$-tough if and only if~$G$ contains an edge ${e=uv}$ such that both of the following conditions are met:
    \begin{enumerate}[label=\rm{(\arabic*)}]
        \item \label{thm311} ${\kappa_G(u,v)\geq 2t+1}$. 
        \item \label{thm312} Every separator~$S$ in~$G$ that is also a $u$-$v$ separator in~${G-e}$ satisfies
        \[
            \abs{S} \geq t \cdot (c(G-S) + 1).
        \]
    \end{enumerate} 
\end{theorem}

For convenience, we restate the above theorem in the following equivalent way.

\begin{corollary}\label{mintough}
    Let $G$ be a non-complete graph and let $t$ be the toughness of~$G$. 
    Then $G$ is minimally $t$-tough if and only if for every edge $e=uv$ at least one of the following conditions is met:
    \begin{enumerate}[label=\rm{(\arabic*)}]
        \item\label{cond1} ${\kappa_G(u,v)<2t+1}$, or
        \item\label{cond2} there exists a separator~$S$ in~$G$ that is also a $u$-$v$ separator in~${G-e}$, such that
        \[
            \abs{S} < t \cdot (c(G-S)+1).
        \]
    \end{enumerate} 
\end{corollary}

Next, we show that Condition~\ref{cond1} of the above corollary is equivalent to two other conditions. 

\begin{lemma}\label{Cond1-equiv}
    Let~$G$ be a non-complete graph, let~$t$ be the toughness of~$G$, and let~$uv$ be an arbitrary edge in~$G$. 
    Then the following are equivalent:
    \begin{enumerate}[label=\rm{(\roman*)}]
        \item\label{Cond1i} $\kappa_G(u,v) < 2t+1$,
        \item\label{Cond1ii} $\kappa_G(u,v) < \ceil{2t}+1$,
        \item\label{Cond1iii} $\kappa_G(u,v) = \ceil{2t}$.
    \end{enumerate}
\end{lemma}

\begin{proof}
    That \ref{Cond1i} implies \ref{Cond1ii} is straightforward, since $2t \leq \ceil{2t}$.
    
    Next, we show that \ref{Cond1ii} implies \ref{Cond1iii}. Assume $\kappa_G(u,v) < \ceil{2t}+1$. 
    By \zcref{toughness-connectivity}, $\ceil{2t} \leq \kappa(G)$. 
    By \zcref{conn-ivdp}, $\kappa(G) \leq  \kappa_G(u,v)$. 
    Hence, $\ceil{2t} \leq \kappa(G) \leq \kappa_G(u,v) < \ceil{2t}+1$ and so $\kappa_G(u,v)=\ceil{2t}$, since $\kappa_G(u,v)$ is an integer.
    
    That \ref{Cond1iii} implies \ref{Cond1i} follows straightforwardly from properties of the ceiling function.
\end{proof}

With \zcref{mintough} we obtain a sufficient condition for a regular graph to be minimally tough.
\begin{lemma}\label{2t-reg}
 Let $G$ be a $\ceil{2t}$-regular graph, where $t$ is the toughness of $G$.
 Then $G$ is minimally tough.
\end{lemma}
\begin{proof}
Notice that $G\ncong K_n$ for any $n\geq 1$, since $\tau(K_n)=\infty$ but 
$K_n$ is a $(n-1)$-regular graph.
If $t=0$, then $G$ is $0$-regular, non-complete graph, hence, $G$ is minimally tough. 
If $t>0$, then $G$ is a connected, non-complete graph. In particular, $\abs{E(G)}\geq 2$.
Let $e=uv$ be an arbitrary edge of $G$. By \zcref{conn-ivdp}, $\kappa_G(u,v) \leq \min\{d_G(u),d_G(v)\}=\ceil{2t}<2t+1$. 
Hence, Condition~\ref{cond1} of \zcref{mintough} is satisfied for all edges of $G$. 
Thus $G$ is minimally tough.
\end{proof}

\subsection{Minimal toughness of graphs with dominating edges}
A \emph{dominating edge} in a graph $G$ is an edge $e=uv\in E(G)$ such that $N(u) \cup N(v)=V(G)$.
We show that for dominating edges, Condition~\ref{cond2} of \zcref{mintough} is never satisfied.
This is helpful in our characterizations of minimally tough graphs containing dominating edges, such as joins of graphs, complete multipartite graphs, $P_4$-free graphs, complements of forests, and co-chordal graphs with complements of diameter at least~$3$.%

We first characterize dominating edges in two different ways.

\begin{lemma}\label{dom-edge}
    For every edge~$uv$ in a graph~$G$, the following conditions are equivalent. 
    \begin{enumerate}
        \item\label{item:dom-edge} $uv$ is a dominating edge in~$G$.
        \item\label{item:dom-edge-sep} ${S \cap \{u,v\} \neq \emptyset}$ for every separator~$S$ in~$G$.
        \item\label{item:dom-edge-complement-distance3} ${d_{\Gc}(u,v) \geq 3}$.
    \end{enumerate}
\end{lemma}

\begin{proof}
    Suppose first that~$uv$ is a dominating edge in~$G$ and let~$S$ be an arbitrary separator of~$G$.
    Suppose for a contradiction that~${S \cap \{u,v\} = \emptyset}$. 
    Since the graph~${G[\{u,v\}]}$ is connected and~${S \cap \{u,v\} = \emptyset}$, the vertices~$u$ and~$v$ belong to the same component~$C$ of~${G-S}$.
    Furthermore, since $S$ is a separator of~$G$, there exists another component of~${G-S}$, say $C'$. 
    Let~$w$ be an arbitrary vertex of~$C'$. 
    Then~${w \in V(G) = N(u) \cup N(v)}$.
    We may assume without loss of generality that~${w \in N(u)}$, implying that~$w$ belongs to~$C$, a contradiction. 
    
    Next, suppose that ${S \cap \{u,v\} \neq \emptyset}$ for every separator~$S$ in~$G$.
    Suppose for a contradiction that ${d_{\Gc}(u,v)\leq 2}$. 
    Since~$u$ and~$v$ are adjacent in~$G$, they are not adjacent in~$\Gc$.
    Hence, ${d_{\Gc}(u,v) = 2}$; in particular, $u$ and~$v$ have a common neighbour, say~$x$, in~$\Gc$. 
    Since vertex~$x$ is adjacent to neither~$u$ nor~$v$ in~$G$, the set~${S \coloneqq V(G) \setminus \{x,u,v\}}$ is a separator in~$G$ such that~${S \cap \{u,v\} = \emptyset}$, a contradiction.
    
    Finally, suppose that~${d_{\Gc}(u,v)\geq 3}$. 
    Then~${uv \notin E(\Gc)}$ and~${N_{\Gc}(u) \cap N_{\Gc}(v)=\emptyset}$, or equivalently, ${N_{\Gc}[u] \cap N_{\Gc}[v]=\emptyset}$.
    Hence, ${(V(G)\setminus N_G(u))\cap (V(G)\setminus N_G(v))=\emptyset}$. 
    So ${V(G)\setminus (N_G(u) \cup N_G(v))=\emptyset}$, which implies ${N_G(u) \cup N_G(v)=V(G)}$, that is, $uv$ is a dominating edge in~$G$.  
\end{proof}

\begin{corollary}\label{cor:dom2}
    If~${e=uv}$ is a dominating edge in a graph~$G$, then no separator in~$G$ is also a $u$-$v$ separator in~${G-e}$.
\end{corollary}

\begin{proof}
    By \zcref{dom-edge}, every separator of~$G$ contains~$u$ or~$v$, hence it is not a $u$-$v$ separator in~${G-e}$.
\end{proof}

Next, we relate the local connectivity of a dominating edge in a minimally tough graph with the toughness of the graph.

\begin{lemma}\label{dom-edge-cond1}
    Let~${t}$ be a real number, let~$G$ be a minimally $t$-tough graph, and let~$uv$ be a dominating edge in~$G$. 
    Then~${\kappa_G(u,v) < 2t+1}$.
\end{lemma}

\begin{proof}
Since $G$ is minimally tough, at least one of Conditions~\ref{cond1} or~\ref{cond2} of \zcref{mintough} holds for the edge $e=uv$. But since $uv$ is a dominating edge, by \zcref{cor:dom2}, there is no separator in $G$ that is also a $u$-$v$ separator in $G-e$, hence Condition~\ref{cond2} of \zcref{mintough} does not hold. So, Condition~\ref{cond1} must hold, i.e., $\kappa_G(u,v) < 2t+1$.
\end{proof}

\zcref{Cond1-equiv,conn-ivdp,toughness-connectivity,dom-edge-cond1} imply the following.

\begin{corollary}\label{cor:dom-edge-cond1}
 Let~${t}$ be a real number, let $G$ be a minimally $t$-tough graph, and let $uv$ be a dominating edge in $G$. 
    Then $\kappa(G) = \kappa_G(u,v)= \ceil{2t}$. 
\end{corollary}

As a first concrete application of dominating edges to minimal toughness, we show that complete graphs are the only minimally tough graphs that can have more than one \emph{universal vertex}, that is, a vertex adjacent to all other vertices.

\begin{lemma}\label{twoUniversalVertices}
Every non-complete minimally tough graph has at most one universal vertex.
\end{lemma}
\begin{proof}
Let $G$ be a non-complete minimally tough graph of order $n$ and let $t$ be the toughness of $G$. Suppose that $G$ has at least two universal vertices, say $u$ and $v$. Then $\abs{N_G(u) \cap N_G(v)}=n-2$ and so $\kappa_G(u,v)=n-1$. Since $G$ is not complete, $\kappa(G) \leq n-2$ and so, by \zcref{toughness-connectivity}, $\ceil{2t}\leq n-2$. 
On the other hand, $uv$ is a dominating edge, hence, by \zcref{cor:dom-edge-cond1}, $\kappa_G(u,v) =\ceil{2t}$. This gives us $n-1=\kappa_G(u,v)=\ceil{2t} \leq n-2$, a contradiction.
\end{proof}

\subsection{Toughness of complete multipartite graphs}

In 1973, Chv\'{a}tal determined the toughness of complete bipartite graphs.

\begin{proposition}[Chv\'{a}tal~{\cite{Chva73}}]\label{toughness-complete-bipartite} 
For any two nonnegative integers~$m\le n$ with $n\ge 2$, the toughness of the complete bipartite graph $K_{m,n}$ is equal to $\frac{m}{n}$.
\end{proposition}

Using the characterization of dominating edges given by \zcref{dom-edge}, we now generalize Chv\'{a}tal's result by determining the toughness of arbitrary complete multipartite graphs.
Each complete multipartite graphs with at least two parts is a join of two smaller graphs; in particular, every such graph contains a dominating edge.

\begin{observation}\label{join-edge-is-dominating}
Let $G_1$ and $G_2$ be arbitrary disjoint graphs and $u \in V(G_1)$ and $v \in V(G_2)$ be arbitrary vertices. 
Let $G=G_1*G_2$.  
Then $N_G(u)=N_{G_1}(u) \cup V(G_2)$ and $N_G(v)=N_{G_2}(v) \cup V(G_1)$. In particular, $uv$ is a dominating edge in $G$.
\end{observation}

We will also need the following lemma. 

\begin{lemma}\label{lem:sep}
Let $G$ be a non-complete graph such that the complement $\overline{G}$ is disconnected and let $V_1,\ldots,V_k$ be the vertex sets of the components of $\overline{G}$. 
Let $S$ be an arbitrary separator of $G$. 
Then, there exists $i \in \{1,\ldots, k\}$ such that $V_j \subseteq S$ for all $ j \neq i$.
\end{lemma}

\begin{proof}
Suppose for a contradiction that there are two distinct parts $V_i$ and $V_j$ that are not fully included in $S$. 
Then, there are two vertices $u \in V_i$ and $v \in V_j$ such that $u,v \notin S$. 
Observe that $G$ is the join of the graphs $G_1\coloneqq G[V_i]$ and $G_2\coloneqq G-V_i$.
Since $u\in V(G_1)$ and $v\in V(G_2)$, by \zcref{join-edge-is-dominating}, $uv$ is a dominating edge in $G$. 
Hence, by \zcref{dom-edge}, $S\cap \{u,v\} \neq \emptyset$, a contradiction.  
\end{proof}

With the above lemma we can now determine the toughness of complete multipartite graphs, generalizing \zcref{toughness-complete-bipartite}.

\begin{proposition}\label{prop3} 
 Let~$n = \sum_{i=1}^{k} n_i$ with~${1\leq n_1 \leq \ldots \leq n_k}$ and~$n_k>1$. 
 Then the toughness of the complete multipartite graph $K_{n_1,\ldots,n_k}$ is given by
 \[\tau(K_{n_1,\ldots,n_k})=\frac{n}{n_k}-1\,.\]
\end{proposition}

\begin{proof} 
Let~${G = K_{n_1,\ldots,n_k}}$.
Since~${n_k>1}$, the graph~$G$ is not complete.
Furthermore, if~${k = 1}$, then~$G$ is disconnected, hence~${\tau(G) = 0}$ and the formula holds, since~${n = n_1}$. 
So let~${k \geq 2}$ and let ${V_1, \ldots, V_k}$ be the parts of~$G$ such that~${\abs{V_i}=n_i}$ for all ${i \in [k]}$. 
Let $S$ be a separator of $G$. 
Since $V_1,\ldots,V_k$ are the vertex sets of the components of $\overline{G}$, by \zcref{lem:sep}, $S$ must contain all the vertices of all but one part, say $V_i$. 
From this it follows that $\abs{S}\geq n-n_i\geq n-n_k$.
Observe also that $c(G-S)\leq n_i\leq n_k$. 
Hence $\frac{\abs{S}}{c(G-S)} \geq \frac{n-n_k}{n_k}=\frac{n}{n_k}-1$ for all separators $S$. 
Furthermore, equality is attained for $S=V(G)\setminus V_k$. 
Thus, by~\eqref{toughness}, $\tau(G)=\frac{n}{n_k}-1$. 
\end{proof}

\section{Joins of graphs}
\label{sec:joins}

In this section we explore minimal toughness of joins of graphs.
We first prove the join condition (\zcref{joincondition}): if one of the two graphs contains a vertex of maximum degree in the join, then the other graph must be regular.
We restate it here in greater detail. 

\begin{theorem}\label{G2-regular} 
Let~${t}$ be a real number and let~$G$ be a minimally $t$-tough graph such that ${G = G_1*G_2}$ is the join of graphs $G_1$ and~$G_2$. 
Let $t_2$ be the toughness of $G_2$. 
Suppose furthermore that~$G_1$ contains a vertex of maximum degree in~$G$. 
Then $G_2$ is a $\ceil{2t_2}$-regular, minimally tough graph and ${\ceil{2t}=\ceil{2t_2}+\abs{V(G_1)}}$.
\end{theorem}

\begin{proof}
Let $n_1=\abs{V(G_1)}$.
Let $v \in V(G_1)$ be a vertex of maximum degree in $G$, i.e., $d_{G}(v)=\Delta(G)$. 
Let $u$ be an arbitrary vertex of~$G_2$. 
Since $uv$ is a dominating edge in $G$ (by \zcref{join-edge-is-dominating}), we get, by \zcref{cor:dom-edge-cond1}, that ${\kappa_{G}(u,v)=\ceil{2t}}$. 
On the other hand, by \zcref{lem:kuv}, ${\kappa_{G}(u,v)=\min\{d_{G}(u),d_{G}(v)\}=d_{G}(u)}$. 
Notice that $d_{G}(u)=d_{G_2}(u)+n_1$, and so 
\begin{equation}\label{eq:g2reg}
d_{G_2}(u)+n_1 =\ceil{2t}.
\end{equation}

Notice that~$G_2$ is non-complete, as otherwise, each vertex in~$G_2$ would be universal in~$G$, and since then~$v$ is also universal in~$G$, this would contradict \zcref{twoUniversalVertices}.
Hence, there exists some separator in~$G_2$. 
Let~$S_2$ be a tough separator in~$G_2$. 
Then ${t_2=\frac{\abs{S_2}}{c(G_2-S_2)}}$. 
Furthermore, ${S=S_2 \cup V(G_1)}$ is a separator in $G$ and ${G-S=G_2-S_2}$, implying ${c(G-S)=c(G_2-S_2)}$. 
Hence \[{t \leq  \frac{\abs{S}}{c(G-S)}=\frac{\abs{S_2}+n_1}{c(G_2-S_2)}=t_2+\frac{n_1}{c(G_2-S_2)}}.\] 
And since $c(G_2-S_2) \geq 2$, we get 
\begin{equation}\label{tt_2}
t \leq t_2 + \frac{n_1}{2}\,.
\end{equation}
Combining~\eqref{eq:g2reg} and \eqref{tt_2} yields ${d_{G_2}(u) + n_1 =\ceil{2t} \leq \ceil{2t_2 +n_1}=\ceil{2t_2}+n_1}$, hence ${d_{G_2}(u) \leq \ceil{2t_2}}$.

On the other hand, by \zcref{toughness-connectivity} and 
since~$\kappa(G_2) \leq \delta(G_2)$, it follows that \[{{\ceil{2t_2} \leq  \kappa(G_2) \leq \delta(G_2) \leq d_{G_2}(u)}}.\] 
We thus have $\ceil{2t_2} \leq d_{G_2}(u)\leq \ceil{2t_2}$, so $d_{G_2}(u)=\ceil{2t_2}$ for every $u \in V(G_2)$. 
Hence, $G_2$ is $\ceil{2t_2}$-regular and, by \zcref{2t-reg}, minimally tough.
Combining the equality $d_{G_2}(u)=\ceil{2t_2}$ with \eqref{eq:g2reg} 
 gives us $\ceil{2t}=\ceil{2t_2}+n_1$, as claimed.
\end{proof}

As an easy consequence, we observe the following. 

\begin{remark}\label{G1G2-regular}
 If in \zcref{G2-regular} both $G_1$ and $G_2$ contain a vertex of maximum degree in $G=G_1*G_2$, then both $G_1$ and $G_2$ must be regular.
\end{remark}

Next, we derive an important result from \zcref{G2-regular}, reducing the problem of characterizing certain minimally tough joins of graphs to a particular infinite family of graphs.
This result will be the main tool for proof of \zcref{P4free:mintough} (in \zcref{sec:P4free}) and for one of our partial characterizations of minimally tough co-chordal graphs (see \zcref{characterization-co-chordal-disconnected}).

\begin{lemma}\label{MinTough-Join-Disconnected} 
Let $G$ be a minimally tough graph such that $G=G_1*G_2$ is the join of graphs $G_1$ and~$G_2$ such that  $G_1$ contains a vertex of maximum degree in $G$ and $G_2$ is disconnected.
Then~$G$ is isomorphic to one of
$K_{1,\ell}$, 
$K_{2,3}$,
$\turan{2\ell}{\ell}$, or
$\turan{2\ell-1}{\ell}$ for some~${\ell \geq 2}$.
\end{lemma}

\begin{proof}
Let $n_1=\abs{V(G_1)}$ and $n_2=\abs{V(G_2)}$. 
Since $G_2$ is disconnected, $n_2 \geq 2$.
Let $t$~be the toughness of $G$.
Notice that $\tau(G_2)=0$, since $G_2$ is disconnected.
Hence, by \zcref{G2-regular}, $G_2$ is $0$-regular, i.e., edgeless.

If $n_1=1$, then, since $n_2 \geq 2$, we infer that $G\cong
K_{1,n_2}$.
So from now on, we assume that $n_1\ge 2$.
Let $v\in V(G_1)$ be a vertex of maximum degree in $G$ and $u \in V(G_2)$ be an arbitrary vertex. 
First, ${d_G(v)=n_2+d_{G_1}(v)}$, and since $v$ is a vertex of maximum degree in $G$, it must also be a vertex of maximum degree in $G_1$, so ${d_G(v)=n_2+d_{G_1}(v)=n_2 + \Delta(G_1)}$. 
Furthermore, since $G_2$ is edgeless, we have $d_G(u)=n_1$, and the inequality ${d_G(u) \leq d_G(v)}$ is equivalent to
\begin{equation}\label{eq:n1}
n_1 \leq n_2 + \Delta (G_1).
\end{equation}
Furthermore, by \zcref{lem:kuv}, 
$\kappa_G(u,v)=d_G(u)=n_1$. 

Second, since $G_2$ is disconnected, the set $S \coloneqq V(G_1)$ is a separator in $G$ and ${c(G-S)=n_2}$, hence ${t \leq \frac{n_1}{n_2}}$.
Furthermore, $uv$ is a dominating edge in $G$ and so, by \zcref{dom-edge-cond1},
\[n_1 = \kappa_G(u,v) < 2t +1 \leq 2\frac{n_1}{n_2} +1\,.\] 
Since $n_1\ge 2$, rearranging the inequality $n_1 < 2\frac{n_1}{n_2} +1$ yields
\begin{equation}\label{eq:n2}
n_2 < \frac{2n_1}{n_1-1}=2+\frac{2}{n_1-1}\le 4\,.
\end{equation}
Hence $n_2 \in \{2,3\}$.
We consider the two cases separately.
 
\emph{Case 1:} $n_2=3$. 
In this case the first inequality in~\eqref{eq:n2} becomes $3 < 2+ \frac{2}{n_1-1}$, and rearranging it we obtain $n_1<3$. 
Hence $n_1=2$. 
Since, by \zcref{twoUniversalVertices}, $G$ cannot have two universal vertices, we have $G_1 \cong 2K_1$. 
Hence $G \cong K_{2,3}$.

\emph{Case 2:} $n_2=2$. 
In this case inequality~\eqref{eq:n1} becomes $n_1 \leq \Delta(G_1) + 2$ and consequently $n_1-2 \leq \Delta(G_1) \leq n_1-1$. 

If $\Delta(G_1)=n_1-2$, then we have equality in~\eqref{eq:n1}, hence $G_2$ also contains a vertex of maximum degree in $G$. 
As observed in \zcref{G1G2-regular}, this implies that  $G_1$ is regular. 
But then, $\Delta(G_1)=n_1-2$, its complement $\overline{G_1}$ is a $1$-regular graph. 
Hence, $n_1$ must be even and $\overline{G_1} \cong \frac{n_1}{2}K_2$. 
It follows that $G_1 \cong \turan{n_1}{n_1/2}$.
Hence, $G \cong \turan{n_1+2}{n_1/2+1} = \turan{2\ell}{\ell}$, where $\ell\coloneqq\frac{n_1}{2}+1\geq 2$.

On the other hand, if $\Delta(G_1)=n_1-1$, then $v$ is a universal vertex in $G_1$. Furthermore, $v$ is also a universal vertex in $G$ since it is adjacent to both vertices of $G_2$. 
We can thus consider $G$ as the join $G[\{v\}]*(G-v)$. 
Clearly, $G[\{v\}]$ contains a vertex of maximum degree in $G$, hence, by \zcref{G2-regular}, $G-v$ is regular. 
Notice that $d_{G-v}(u)=n_1-1$, since $d_G(u)=n_1$. Hence, $G-v$ is an $(n_1-1)${\dash}regular graph. 
Now, let $H:=(G-v)-V(G_2)=G_1-v$ and let $w \in V(H)$ be  arbitrary. Then $d_H(w)=d_{G-v}(w)-2=n_1-3$, since $N_{G-v}(w)=N_H(w)\cup V(G_2)$ and $\abs{V(G_2)} = 2$.
Hence, $H$ is an $(n_1-3)$-regular graph with $n_1-1$ vertices, so $\overline{H}$ is a $1$-regular graph. 
Consequently, $n_1-1$ must be even and $\overline{H} \cong \frac{n_1-1}{2}K_2$. 
Thus, $H \cong \turan{n_1-1}{(n_1-1)/2}$ and so $G_1 \cong \turan{n_1}{(n_1+1)/2}$. 
Hence, $G\cong \turan{n_1}{(n_1+1)/2}*(2K_1)\cong \turan{n_1+2}{(n_1+3)/2}=\turan{2\ell-1}{\ell}$, where $\ell\coloneqq\frac{n_1+3}{2}\geq 3$.
(Observe that for $\ell = 2$, the graph $\turan{2\ell-1}{\ell} = \turan{3}{2}$ is isomorphic to $K_{1,2}$, which was already considered in the case $n_1 = 1$.)
\end{proof}

\subsection{Graphs with a universal vertex}
\label{subsec:cmp}

We established in \zcref{twoUniversalVertices} that every non-complete minimally tough graph has at most one universal vertex. 
Since every nontrivial graph with a universal vertex can be expressed as a join of two smaller graphs, \zcref{G2-regular} leads to the following characterization of minimally tough graphs with universal vertices.

\begin{theorem} \label{join-with-K1}
Let~${t}$ be a real number, let $G$ be a graph with toughness $t$ such that $G$ has a universal vertex $v$, and let $t'$ be the toughness of the graph $G'\coloneqq G-v$.
Then $G$ is minimally tough if and only if $G'$ is a $\ceil{2t'}$-regular graph and $\ceil{2t}=\ceil{2t'}+1$.
\end{theorem}

\begin{proof}
Note that $G$ is the join of the graphs $G[\{v\}]$ and $G'$.

Suppose first that $G$ is minimally tough.
Since $v$ is a vertex of maximum degree in $G$, by \zcref{G2-regular}, $G'$ is $\ceil{2t'}$-regular and $\ceil{2t}=\ceil{2t'}+1$.

Conversely, assume $G'$ is a $\ceil{2t'}$-regular graph and $\ceil{2t}=\ceil{2t'}+1$. 
We show that $G$ is minimally $t$-tough by verifying that Condition~\ref{cond1} of \zcref{mintough} is satisfied for every edge of $G$.
First, we consider edges with both endpoints in $V(G')$. 
So let $uw \in E(G')$ be an arbitrary edge. 
Then $\kappa_{G}(u,w)=\kappa_{G'}(u,w)+1$, since a maximum collection of internally vertex-disjoint $u$-$w$ paths in $G'$, augmented by the path $(u,v,w)$, yields a maximum collection of internally vertex-disjoint $u$-$w$ paths in $G$. 
Hence, by \zcref{conn-ivdp}, \[\kappa_{G}(u,w)=\kappa_{G'}(u,w)+1 \leq \min\{d_{G'}(u),d_{G'}(w)\}+1=\ceil{2t'}+1=\ceil{2t}\,,\] so in this case Condition~\ref{cond1} of \zcref{mintough} holds. 
Secondly, we consider edges incident with $v$. 
So let $u \in V(G')$ be arbitrary. 
Then for the edge $uv$ we have $\kappa_{G}(u,v)= d_{G'}(u)+1=\ceil{2t'}+1=\ceil{2t}$, and so Condition~\ref{cond1} of \zcref{mintough} is met also for this type of edges.
Hence, by \zcref{mintough}, $G$ is minimally $t$-tough.
\end{proof}

Dallard et al.~showed in \cite{Dall23} that for a rational number $t\leq 1$ the only minimally $t$-tough graphs with universal vertices are the stars $K_{1,\ell}$ for $\ell\geq 2$ (see \cite[Theorem~4.6]{Dall23}).
Note that the toughness of $K_{1,\ell}$ is $1/\ell$ (see~\zcref{toughness-complete-bipartite}).
Dallard et al.~also observed that wheel graphs are examples of minimally $t$-tough graphs with a universal vertex for $t>1$. 
For an integer $\ell\ge 5$, the \emph{wheel} $W_{\ell}$ is the join of $K_1$ and cycle $C_{\ell-1}$.
Dallard et al.~\cite[Proposition~4.7]{Dall23} proved that for all $\ell \ge 5$, the wheel graph $W_{\ell}$ is a minimally tough graph, with toughness equal to $1 + 2/(\ell-1)$ if $\ell$ is odd, and $1 + 2/\ell$ if $\ell$ is even.
 
Using \zcref{join-with-K1} we extend their characterization of minimally $t$-tough graphs with universal vertices for $t\leq 1$ to $t\leq 3/2$, by showing in addition that the wheel graphs are the only minimally $t$-tough graphs with universal vertices for $t\in (1,\frac{3}{2}]$.

\begin{proposition}\label{UniversalVertex-characterization} 
Let~$t \in (0, 3/2]$ and let $G$ be a minimally $t$-tough graph with a universal vertex. 
Then one of the following holds:
\begin{enumerate}[label=\rm{(\arabic*)}]
\item\label{UniversalVertex0} $t\le 1/2$ and $G \cong K_{1,\ell}$ for some $\ell \geq 2$, or
\item\label{universal vertex2} $t>1$ and $G \cong W_{\ell}$ for some $\ell \geq 5$.
\end{enumerate}
\end{proposition}

\begin{proof}
Let $v$ be a universal vertex in $G$. Let $G'=G-v$ and let $t'$ be the toughness of $G'$. 
By \zcref{join-with-K1}, $G'$ is $\ceil{2t'}$-regular and $\ceil{2t'}=\ceil{2t}-1$. 
We consider three cases.

Suppose first that $t \in (0, 1/2]$. 
Then $\ceil{2t'}=\ceil{2t}-1=0$, hence $G'$ is $0$-regular, so $G' \cong \ell K_1$ for some $\ell \geq 2$, implying that $G \cong K_{1, \ell}$.

Suppose now that $t \in (1/2,1]$. 
Then $\ceil{2t'}=\ceil{2t}-1=1$, hence $G'$ is $1$-regular. 
So $G' \cong m K_2$ for $m \geq 2$ ($m\neq 1$, since $G'$ is not complete). 
But then $G'$ is disconnected, hence $t'=0$, a contradiction with $\ceil{2t'}=1$. So for $t \in (1/2,1]$, there are no minimally $t$-tough graphs with a universal vertex.

Assume $t \in (1, 3/2]$. 
Then $\ceil{2t'}=\ceil{2t}-1=2$, hence $G'$ is $2$-regular, and since $t'>0$, $G'$ is connected. Furthermore, $G'$ is not complete. 
Hence, $G' \cong C_m$ for some $m \geq 4$ and consequently $G \cong W_{m+1}$.
\end{proof}

Finally, we use \zcref{join-with-K1} to give a new proof of the following result from~\cite{Dall23}.
\begin{theorem}[Dallard et al.~{\cite[Theorem~4.8]{Dall23}}]
For every real number $t > 1$, there exists no minimally $t$-tough chordal graph with a universal vertex.
\end{theorem}
\begin{proof}
Let ${t>1}$ be a real number and suppose that there exists a minimally $t$-tough chordal graph~$G$ with a universal vertex~$v$.
Let ${G'= G-v}$ and let~$t'$ be the toughness of~$G'$. 
By \zcref{join-with-K1}, $G'$ is $\ceil{2t'}$-regular and ${\ceil{2t'}=\ceil{2t}-1}$. 
In particular, since~${t>1}$, we get $\ceil{2t'}\geq 2$, so $G'$ is a connected $d$-regular graph, for some $d\geq 2$. 
Next, observe that, since $G$ is chordal, so is $G'$.
By \zcref{regular-chordal}, $G' \cong K_{d+1}$, hence, $G\cong K_{d+2}$.
Therefore, $t=\infty$, a contradiction.
\end{proof}

\section{%
\texorpdfstring{$P_4$-free graphs}{P\_4-free graphs}}
\label{sec:P4free}
In this section, we characterize minimally tough $P_4$-free graphs. 
To do so, we need the following theorem describing the structure of connected $P_4$-free graphs.

\begin{theorem}[Corneil et al.~\cite{Corn81}, Gurvich~\cite{Gurv77, Gurv84},  Sumner~\cite{Sumn71}] \label{P4freechar}
Let $G$ be a connected $P_4$-free graph with $\abs{V(G)} \geq 2$. 
Then, there exists a unique partition of $V(G)$ into a maximum number of $k\geq 2$ parts $V_1,\ldots, V_k$, such that the following conditions are satisfied.
\begin{enumerate}[label=\rm{(\arabic*)}]

\item\label{P4c3} For every $i\in\{1,\ldots, k\}$, either $G[V_i] \cong K_1$ or $G[V_i]$ is disconnected. 
\item\label{P4c2} For every $u \in V_i$ and  $v \in V_j$ with $i \neq j$, we have $\{u,v\} \in E(G)$.
\end{enumerate}
\end{theorem}

We can now prove \zcref{P4free:mintough}; in fact, we prove the following strengthening of it.

\begin{theorem}\label{P4free:mintough-new}
Let $G$ be a $P_4$-free graph.
Then, the following conditions are equivalent. 
\begin{enumerate}
    \item\label{item:P4free:1} $G$ is non-trivially minimally tough. 
    \item\label{item:P4free:2} $G$ is isomorphic to one of~$K_{1,\ell}$, $K_{2,3}$, $\turan{2\ell}{\ell}$, or $\turan{2\ell-1}{\ell}$ for some~${\ell \geq 2}$. 
    \item\label{item:P4free:3} $G$ is an $n$-vertex complete multipartite graph $K_{n_1, \dots, n_k}$ with $k\geq2$ parts such that ${n_1 \leq \dots \leq n_k}$ and~${n - n_2 < \frac{2n}{n_k}}-1$. 
\end{enumerate}
\end{theorem}

\begin{proof}
We may assume that $G$ is connected and has at least two vertices, since otherwise neither of the three conditions holds.
By \zcref{P4freechar} there is a unique partition of $V(G)$ into maximum number $k\geq 2$ of parts $V_1,\ldots, V_k$ satisfying Conditions \ref{P4c3} and \ref{P4c2} of that theorem. 

Let us show that \ref{item:P4free:1} implies \ref{item:P4free:2}.
Assume that $G$ is non-trivially minimally tough. 
Let $u$ be a vertex of maximum degree in $G$, let $i\in [k]$ be the index of the part containing $u$ (that is, $u\in V_i$), and fix an arbitrary $j\in [k]\setminus \{i\}$.
Then, $G[V_j]$ is disconnected, since otherwise the unique vertex in $V_j$ would be a universal vertex in $G$, implying by the choice of $u$ that $u$ is also universal, contradicting \zcref{twoUniversalVertices}.
Observe that $G$ is the join of graphs $G_1\coloneq G-V_j$ and $G_2 \coloneq G[V_j]$.
Since $u\in V(G_1)$ and $G_2$ is disconnected, \zcref{MinTough-Join-Disconnected} implies that $G$ is isomorphic to one of
$K_{1,\ell}$, 
$K_{2,3}$,
$\turan{2\ell}{\ell}$, or
$\turan{2\ell-1}{\ell}$ for some~${\ell \geq 2}$.

The fact that \ref{item:P4free:2} implies \ref{item:P4free:3} follows from a simple calculation that is left to the reader. 

To see that \ref{item:P4free:3} implies \ref{item:P4free:1}, assume that $G$ satisfies the conditions of \ref{item:P4free:3}.
Let~$t$ be the toughness of~$G$.
Since $G$ is not complete, $n_k>1$; hence, by \zcref{prop3}, ${t=\frac{n}{n_k}-1}$. 
Observe that every edge in a complete multipartite graph is dominating. 
We prove that $G$ is minimally tough by applying \zcref{mintough}.
By \zcref{cor:dom2}, the second condition of \zcref{mintough} is never met, hence, it suffices to show that~${\kappa_G(u,v) < 2t + 1}$ for every edge~${uv \in E(G)}$. 

Let $uv \in E(G)$ be an arbitrary edge, and let $V_i$ and $V_j$, with $1\leq i<j\leq k$, be the parts that contain $u$ and $v$, respectively. 
Observe that $G$ is the join of graphs $G_1\coloneq G[V_i]$ and $G_2 \coloneq G-V_i$. 
Furthermore, since the parts of the complete multipartite graph $G$ are exactly the sets $V_1,\ldots, V_k$, we infer that ${d_G(u) = n-n_i\ge n-n_j = d_G(v)}$.
Since  $u \in V(G_1)$ and $v \in V(G_2)$, by \zcref{lem:kuv}, $\kappa_G(u,v) = n-n_j$.
Since, by assumption, $n_2 \leq n_j$ and ${n - n_2 < \frac{2n}{n_k}-1}$, it follows that $\kappa_G(u,v)\leq n-n_2< \frac{2n}{n_k} - 1=2t+1$, as required.
\end{proof}

Recall that the generalized Kriesell's conjecture states that for every positive real number $t$, every minimally $t$-tough graph has a vertex of degree $\ceil{2t}$.
It follows from \zcref{P4free:mintough-new} that the generalized Kriesell's conjecture holds for $P_4$-free graphs.

\begin{corollary}\label{cor:Kriesell-P4-free}
    Let $t$ be a positive real number and let $G$ be a minimally $t$-tough $P_4$-free graph.
    Then $G$ has a vertex of degree $\ceil{2t}$.
\end{corollary}

\begin{proof}
    Since $t$ is a positive real number, $G$ is non-trivially minimally tough. Hence, by \zcref{P4free:mintough-new}, $G$ is isomorphic to one of
    $K_{1,\ell}$, 
    $K_{2,3}$,
    $\turan{2\ell}{\ell}$, or
    $\turan{2\ell-1}{\ell}$ for some~${\ell \geq 2}$. 
    We now show that in each of the above cases $\delta(G) = \ceil{2t}$.
    
    If $G\cong K_{1,\ell}$, for some $\ell\geq 2$, then, by \zcref{toughness-complete-bipartite}, $t=\frac{1}{\ell}$. Clearly, $\delta(K_{1,\ell})=1$ and, hence, ${\delta(G) = 1 = \ceil{2t}}$.
    If $G\cong K_{2,3}$, then, by \zcref{toughness-complete-bipartite}, ${t=\frac{2}{3}}$. 
    Since $\delta(K_{2,3})=2$, we have ${\delta(G)=2=\ceil{2t}}$. 
    Lastly, if $G$ is isomorphic to $\turan{2\ell}{\ell}$, or
    $\turan{2\ell-1}{\ell}$ for some~${\ell \geq 2}$, then, $G$ is a complete multipartite graph whose largest part has size~$2$, hence, by \zcref{prop3}, ${t=\frac{n}{2}-1}$, where~$n$ is the order of~$G$. 
    Furthermore, it is easy to see that ${\delta(G)=n-2}$. 
    Hence, ${\delta(G)=n-2 = \ceil{2t}}$.
\end{proof}

\section{Co-chordal graphs}
\label{sec:cochordal}
In this section we characterize minimally tough co-chordal graphs with co-diameter at least $3$. 
We start with the case when the co-diameter is infinite, that is, when the complement is disconnected.

\begin{lemma}\label{characterization-co-chordal-disconnected}
    A co-chordal graph with disconnected complement is non-trivially \linebreak minimally tough if and only if it is isomorphic to one of $K_{1,\ell}$, $K_{2,3}$, $\turan{2\ell}{\ell}$, or $\turan{2\ell-1}{\ell}$ for some~${\ell \geq 2}$.
\end{lemma}

\begin{proof}
Let $G$ be a co-chordal graph with disconnected complement.
Let $v \in V(G)$ be a vertex of maximum degree in $G$ and let $Q$ be a connected component of $\Gc$ that does not contain $v$.
Then $G=\overline{Q}*(G-V(Q))$. 
Assume $G$ is non-trivially minimally tough. 
Then, by \zcref{G2-regular}, $\overline{Q}$ is regular.
Furthermore, by \zcref{twoUniversalVertices}, $G$ contains at most one universal vertex and hence, ${\abs{V(Q)} \geq 2}$.
In particular, $Q$ is a regular chordal graph with at least two vertices.
By \zcref{regular-chordal}, $Q$ is complete, implying that the graph~$\overline{Q}$ is edgeless and, hence, disconnected.
So the result follows from \zcref{MinTough-Join-Disconnected}. 

The converse follows from \zcref{P4free:mintough-new}. 
\end{proof} 

Recall by \zcref{SimplVert-MaxDistance} that in every chordal graph with positive (and finite) diameter~$d$ there exist two simplicial vertices at distance $d$. 
In the next lemma we consider the case $d\ge 3$ and determine the local connectivity of such a pair of vertices in the complement of the graph. 

\begin{lemma}\label{cochordalkappa}
    \sloppypar
    Let~$G$ be a co-chordal graph with co-diameter~$d \geq 3$, let~${u,w}$ be simplicial vertices in~$\Gc$ at distance~$d$ in~$\Gc$, let ${U \coloneqq N_{\Gc}(u)}$, let ${W \coloneqq N_{\Gc}(w)}$, let ${X \coloneqq V(G)  \setminus (U \cup W \cup \{u,w\})}$, and let~$m$ be the maximum size of a matching in~${G[U \cup W]}$. 
    Then ${\kappa_G(u,w) = \abs{X} + m+1}$. 
\end{lemma}

\begin{proof}
    Refer to~\zcref{fig:co-chordal} for an illustration of the setup of this lemma. 
    \begin{figure}[htbp]
        \centering
        \includegraphics[width=0.6\linewidth]{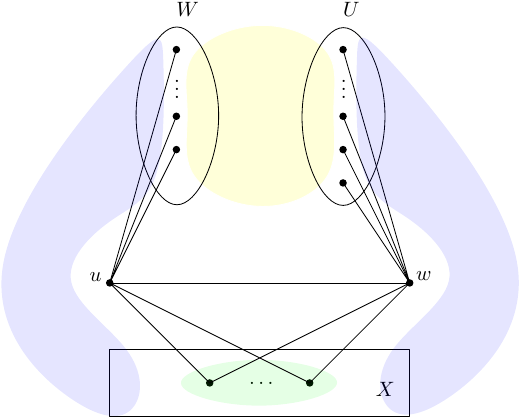}
        \caption{The structure of a co-chordal graph with co-diameter at least~$3$. 
        The figure shows all the edges incident with $u$ or $w$, and $U$ and $W$ are independent sets. 
        In addition, there might be some edges between $U$ and $W$ (yellow), between $U\cup W$ and $X$ (purple), and inside $X$ (green).}\label{fig:co-chordal}
    \end{figure}
    
    Let $\cP^*$ be a maximum collection of pairwise internally vertex-disjoint $u$-$w$ paths in~$G$, such that the \emph{total length} of $\cP^*$, that is, $\sum_{P \in \cP^*} \abs{E(P)}$, is minimal.

    First, observe that $(u,w) \in \cP^*$, by the maximality of~$\cP^*$.
    
    Second, for every ${x \in X}$, it follows from the definition of~$X$ that $(u,x,w)$ is a path in $G$. Furthermore, $(u,x,w) \in \cP^*$. Indeed, if for some $x \in X$ the path $(u,x,w) \notin \cP^*$, then either $x$ is not a vertex of any path in $\cP^*$, or there exists a path $P' \in \cP^*$ of length at least $3$ and containing $x$. In the first case, we get a contradiction with maximality of $\cP^*$, since the path $(u,x,w)$ could be added to $\cP^*$.
    In the second case,  $P'$ could be replaced by the path $(u,x,w)$, which is in contradiction with minimality of the total length of $\cP^*$.

    Lastly, let $\cP_3^*$ be the remaining paths in $\cP^*$, that is, 
    $\cP_3^*$ is the set of paths in $\cP^*$ that do not belong to the set $\{(u,w)\}\cup\{(u,x,w) \colon x \in X\}$.
    Then $\cP_3^*$ is a maximum collection of pairwise internally vertex-disjoint $u$-$w$ paths in the graph ${\widehat{G} \coloneqq G[U\cup W \cup \{u,w\}]-uw}$. Notice that $G[U \cup W]$ is a  bipartite graph, $N_{\widehat{G}}(u) = W$, and $N_{\widehat{G}}(w) = U$.
    Hence, $\widehat{G}$ is the $(u,w)$-extension (see \zcref{extension-of-G}) of $(G[U \cup W],W,U)$ and thus, by \zcref{matching-ivdp}, $\cP_3^*=\cP(M)$ for some maximum matching $M$ in $G[U\cup W]$. Consequently, ${\abs{\cP_3^*} = m}$. 
    
    So ${\kappa_G(u,w) 
    = \abs{\cP^*} 
    = \abs{\{(u,w)\}} + \abs{\{(u,x,w) \colon x \in X\}} + \abs{\cP_3^*}
    = 1 + \abs{X} + m}$.
\end{proof}

Next, we use \zcref{cochordalkappa} to show that there is no minimally tough co-chordal graph with connected complement and co-diameter at least~$4$. 

\begin{lemma}\label{characterization-co-chordal-diam4}
     Let $G$ be a 
     co-chordal graph with connected complement and co-diameter greater or equal to~$4$. 
     Then $G$ is not minimally tough.
\end{lemma}

\begin{proof}
Since $\Gc$ is chordal and ${\diam(\Gc) \geq 4}$, by \zcref{SimplVert-MaxDistance}, there exist simplicial vertices~$u$ and~$w$ at distance $\diam(\Gc)$ in $\Gc$. 
We define $U$, $W$, $X$, and~$m$ as in \zcref{cochordalkappa}, and conclude that ${\kappa_G(u,w) = \abs{X} + m + 1}$. 
Without loss of generality assume that ${\abs{W} \leq \abs{U}}$, and observe that~$G[U \cup W]$ is a complete bipartite graph. 
Hence~$m = \abs{W}$, implying ${\kappa_G(u,w) = \abs{X} + \abs{W} + 1}$.

Let~$t$ be the toughness of $G$. 
Since the distance between $u$ and $w$ in $\Gc$ is at least $4$, by \zcref{dom-edge}, $uw$ is a dominating edge in $G$.
Hence, by \zcref{dom-edge-cond1}, to show that $G$ is not minimally tough, it suffices to show that $2t+1\leq \kappa_G(u,w)$. 

Since $N_{\Gc}(u) = U$, we infer that $U$ is non-empty, as otherwise $\Gc$ would not contain any $u$-$w$ path, contrary to the fact that $\Gc$ is connected.  
Similarly, $W$ is non-empty. 
We consider two cases.
 
\emph{Case 1:} $\abs{U}=1$. Let~${n \coloneqq \abs{V(G)}}$. 
In this case $\abs{W} = 1$ and hence, ${n = \abs{X}+4}$ and ${\kappa_G(u,w)=\abs{X}+2=n-2}$.
Furthermore, since~$\Gc$ is connected and ${\diam(\Gc) \geq 4}$, there exists a vertex ${x \in V(\Gc)}$ with ${d_{\Gc}(x) \geq 2}$. 
Then ${d_G(x) \leq n-3}$ and $N_G(x)$ is a separator in~$G$. 
Hence ${t \leq \frac{n-3}{2}}$. 
So, ${2t+1 \leq n-3+1=\kappa_G(u,w)}$. 

\emph{Case 2:} $\abs{U} \geq 2$. 
Notice that $\abs{X} \geq 1$, since otherwise $\Gc$ would be disconnected, as there are no edges between $U\cup \{u\}$ and $W \cup \{w\}$ in $\Gc$. 
Furthermore, the set ${S=X \cup W \cup \{w\}=N_G(u)}$ is a separator in $G$ and $V(G)\setminus S=U \cup \{u\}$. 
Since $U \cup \{u\}$ is independent in~$G$, we have~${c(G-S)=\abs{U}+1}$ and hence, ${t \leq \frac{\abs{X}+\abs{W}+1}{\abs{U}+1} \leq \frac{\abs{X}+\abs{W}+1}{3}}$, since $\abs{U}\geq 2$.
Thus 
\[
    2t+1 \leq \frac{2}{3}(\abs{X}+\abs{W}+1)+1 
    = \frac{2}{3}\kappa_G(u,w)+1 
    \leq \kappa_G(u,w)\,,
\] 
where the last inequality follows from the fact that ${\kappa_G(u,w) = \abs{X} + \abs{W} + 1 \geq 3}$.
\end{proof}

Finally, we apply~\zcref{cochordalkappa} to show that the only minimally tough co-chordal graphs with co-diameter~$3$ are the double stars. 

\begin{lemma}\label{characterization-co-chordal-diam3}
    Let $G$ be a 
    co-chordal graph with co-diameter~$3$. 
    Then $G$ is minimally tough if and only if $G$ is isomorphic to~$\doublestar{k}{\ell}$ for some~${k,\ell \geq 1}$. 
\end{lemma}

\begin{proof}
Sufficiency of the stated condition follows from the fact that double stars are trees, hence, they are minimally tough.

For necessity, assume that $G$ is minimally tough. 
Since $\Gc$ is chordal with $\diam(\Gc)=3$, by~\zcref{SimplVert-MaxDistance}, there exist simplicial vertices $u$ and $w$ at distance $3$ in $\Gc$. 
Again using the setup of \zcref{cochordalkappa}, we conclude that~${\kappa_G(u,v) = \abs{X} + m + 1}$. 

Let $t$ denote the toughness of $G$.
Since $uw$ is a dominating edge by \zcref{dom-edge}, with~\zcref{dom-edge-cond1} we obtain that $\kappa_G(u,w)<2t+1$, and consequently $\kappa_G(u,w)=\ceil{2t}$ by~\zcref{Cond1-equiv}. 
Hence,
\begin{equation}\label{eq:kGuv}
    \abs{X}+ m + 1 = \ceil{2t}.
\end{equation}
For every pair of vertices $u'\in U$ and $w'\in W$ we get from~\eqref{eq:kGuv}, \zcref{toughness-connectivity}, and~\zcref{conn-ivdp} that
\begin{equation}\label{kappaGu'w'}
\abs{X} + m + 1 = \ceil{2t} \leq \kappa(G) \leq \kappa_G(u',w').
\end{equation}
Furthermore, we define $\cP(u',w')$ to be a maximum collection of internally vertex-disjoint $u'$-$w'$ paths in $G$, and let $\cP_X(u',w') \coloneqq \{P \in \cP(u',w') \colon V(P)\cap X \neq \emptyset\}$ be the set of paths in $\cP(u',w')$ that contain at least one vertex from $X$. 
Since the paths in $\cP(u',w')$ are internally vertex-disjoint,~${\abs{\cP_X(u',w')} \leq \abs{X}}$. 
Let~${\widehat{\cP}(u',w') \coloneqq \cP(u',w') \setminus \cP_X(u',w')}$.
We first establish the following claim.

\medskip\noindent
\emph{Claim.} If ${m=0}$, then ${\abs{X} = 0}$. 
\begin{subproof}[Proof of Claim]
Assume that $m=0$. 
Similarly as in the proof of \zcref{characterization-co-chordal-diam4}, the sets $U$ and $W$ are non-empty.

Let $u' \in U$ and $w'\in W$ be arbitrary vertices. 
Since $m = 0$, there are no edges between $U$ and $W$.
Consequently, the only $u'$-$w'$ path in $G[U\cup \{u\} \cup W \cup \{w\}]$ (i.e., the only possible $u'$-$w'$ path not containing any vertices from $X$) is $(u',w,u,w')$ and hence ${\abs{\widehat{\cP}(u',w')} \leq 1}$.
 But then, by~\eqref{kappaGu'w'}, 
 \[
    \abs{X} + 1 
    \leq \kappa_G(u',w') 
    = \abs{\cP(u',w')}
    = \abs{\cP_X(u',w')} + \abs{\widehat{\cP}(u',w')}
    \leq \abs{X} + 1, 
\]
 which implies that ${\widehat{\cP}(u',w') = \{(w',u,w,u')\}}$ and~${\abs{\cP_X(u',w')} = \abs{X}}$. 

Since ${\abs{\cP_X(u',w')} = \abs{X}}$, each path in $\cP_X(u',w')$ contains exactly one vertex in $X$. 
Furthermore, since ${(w',u,w,u')\in \cP(u',w')}$, no path in $\cP_X(u',w')$ can contain any of the vertices $u$ and $w$. 
Since there are no edges between $U$ and $W$, every path in $\cP_X(u',w')$ of length more than~$2$ must contain at least two vertices from $X$; hence, no such path can exist.
Consequently, all paths in $\cP_X(u',w')$ have length~$2$; in particular, $\cP_X(u',w')=\{(u',x,w') \colon x \in X\}$, since ${\abs{\cP_X(u',w')} = \abs{X}}$. 
This implies that every vertex in $X$ is adjacent in $G$ to both $u'$ and $w'$.

Since the choice of $u' \in U$ and $w' \in W$ was arbitrary, every vertex in $X$ is adjacent in $G$ to every vertex in $U\cup W$.
Furthermore, by definition of $X$, every vertex in $X$ is adjacent in $G$ to both $u$ and $w$.
This implies that $G=(G-X)*X$. 
Hence, if $\abs{X} \geq 1$, then $\overline{G}$ is disconnected, a contradiction with $\diam(\Gc)=3$.  
\end{subproof}

\medskip
If $m=0$, then the above claim implies that $G \cong \doublestar{k}{\ell}$, where $k=\abs{U}\geq 1$ and $\ell=\abs{W}\geq 1$.

For the remainder of the proof, we assume~${m \geq 1}$ and show that there are no minimally tough graphs in that case. 
Since the distance between $u$ and $w$ in $\Gc$ is $3$, there exists a $u$-$w$ path of length $3$ in $\Gc$, say $(u,u_0,w_0,w)$, where $u_0 \in U$, $w_0 \in W$. Then $u_0w_0 \notin E(G)$ and from~\eqref{kappaGu'w'} we get
\begin{equation}\label{kappaGu0w0}
\abs{X} +m +1 \leq \kappa_G(u_0,w_0).
\end{equation}
We consider the sets $\cP \coloneqq \cP(u_0,w_0)$, $\cP_X \coloneqq \cP_X(u_0,w_0)$ and $\widehat{\cP} \coloneqq \widehat{\cP}(u_0,w_0)$ of $u_0$-$w_0$ paths as defined above.
Let $p \coloneqq \abs{\widehat{\cP}}$. 
We have ${\kappa_G(u_0,w_0)=\abs{\cP} = \abs{\cP_X} + \abs{\widehat{\cP}} \leq \abs{X} + p}$. 
This together with inequality~\eqref{kappaGu0w0} gives us $\abs{X}+m+1 \leq \kappa_G(u_0,w_0) \leq \abs{X}+p$. Hence $m+1\leq p.$ 
In particular, since $m\geq 1$, we have $p \geq 2$.

We now show that, with these assumptions, we can always construct a matching $M$ in $G[U \cup W]$ of size at least $p\geq m+1$. 
This will in turn contradict the choice of $m$ as the size of a maximum matching in $G[U \cup W]$.

Let $\cP_{uw} \coloneqq \{P \in \widehat{\cP}\colon V(P) \cap \{u,w\}\neq \emptyset\}$ be the subset of paths in $\widehat{\cP}$ containing at least one of the vertices $u$ or $w$. 
Since the paths in $\widehat{\cP}$ are internally vertex-disjoint, $\abs{\cP_{uw}} \leq 2$.
Notice that every path in $\widehat{\cP} \setminus \cP_{uw}$ has length at least~$3$ and alternates between vertices in $U$ and vertices in $W$.
We consider two cases. 

\medskip
\emph{Case 1:} $ \abs{\cP_{uw}} \leq 1$. 
Since $2\leq p=\abs{\widehat{\cP}} = \abs{\widehat{\cP}\setminus \cP_{uw}} + \abs{\cP_{uw}}$, the assumption $\abs{\cP_{uw}} \leq 1$ implies that $\abs{\widehat{\cP}\setminus \cP_{uw}} \geq 1$.
We construct a matching $M$ in the following way.
We start with $M=\emptyset$. 
Then we take an arbitrary path $Q \in \widehat{\cP} \setminus \cP_{uw}$ and add to $M$ the edges of~$Q$ incident to $u_0$ and $w_0$ (note that these two edges are disjoint, since $Q$ has length at least~$3$).
Finally, for every path  $P \in \widehat{\cP} \setminus (\cP_{uw} \cup \{Q\})$, we add to $M$ any edge of~$P$ not incident to $u_0$ or~$w_0$. 
Since there are at least $p-2$ paths in $\widehat{\cP} \setminus (\cP_{uw} \cup \{Q\})$, each of length at least~$3$, this gives us at least $p-2$ edges in $M$. 
Together with the two edges from path $Q$, we thus get $\abs{M} \geq p$. 
Then, $M$ is indeed a matching in $G[U \cup W]$, since paths in~$\widehat{\cP} \setminus \cP_{uw}$ are internally vertex-disjoint, each of the endpoints $u_0$ and $w_0$ is incident with only one edge in~$M$, and all the endpoints of edges in~$M$ are in $U\cup W$.

\medskip
\emph{Case 2:} $\abs{\cP_{uw}} =2$. 
In this case $\abs{\cP_{uw}}=2$, i.e., there exist distinct paths $P_u$ and $P_w$ in $\cP_{uw}$ such that $u \in V(P_u)$ and $w \in V(P_w)$. 
Furthermore, since $P_u$ and $P_w$ are internally vertex-disjoint, $P_u$ does not contain $w$. 
Hence, there exists a vertex $w_1 \in W\setminus \{w_0\}$ such that $u_0w_1 \in E(P_u)$. 
Similarly, there exists a vertex $u_1 \in U\setminus \{u_0\}$ such that $w_0u_1 \in E(P_w)$. 
We construct a matching~$M$ in $G[U \cup W]$ by taking the edges $u_0w_1$ and $w_0u_1$.
Furthermore, for every path $P \in \widehat{\cP} \setminus \{P_u,P_w\}$, we add to $M$ any edge not incident to~$u_0$ or~$w_0$. 
Then, as before, $M$ is a matching in $G[U \cup W]$ with $\abs{M}=p$. 

\medskip
Hence, in either case, we obtain a contradiction with the fact that $m<p$ and $m$ is the size of a maximum matching in $G[U \cup W]$.
\end{proof}

Combining \zcref{characterization-co-chordal-disconnected,characterization-co-chordal-diam3,characterization-co-chordal-diam4} and the fact that ${\doublestar{1}{1}\cong P_4}$ proves \zcref{characterization-co-chordal}, which we restate for the reader's convenience. 
\getkeytheorem{cochordal}

It follows from \zcref{characterization-co-chordal} that the generalized Kriesell's conjecture holds for co-chordal graphs with co-diameter at least $3$.
\begin{corollary}\label{corollary:GKC2}
    Let $t$ be a positive real number and let $G$ be a minimally $t$-tough co-chordal graph with co-diameter at least $3$.
    Then $G$ has a vertex of degree $\ceil{2t}$.
\end{corollary}

\begin{proof}
    Since~$t$ is a positive real number, $G$ is non-trivially minimally tough.
    Hence, by \zcref{characterization-co-chordal}, $G$ is either $P_4$-free or isomorphic to $\doublestar{k}{\ell}$ for some~${k,\ell \geq 1}$. 
    In the former case the claim holds by \zcref{cor:Kriesell-P4-free}.
    Assume now that ${G\cong \doublestar{k}{\ell}}$ for some ${k,\ell\geq 1}$ and without loss of generality assume ${k\leq \ell}$. 
    Then ${t=\frac{1}{\ell+1}}$ (see, e.g., \cite{Broersma2015}) and ${\delta(G)=1=\ceil{2t}}$.
\end{proof}

\subsection{Net-free co-chordal graphs}
\label{sec:net-free}

We now derive \zcref{characterization-net-free-co-chordal} from \zcref{characterization-co-chordal} and the following characterization of hereditary closed neighbourhood-Helly graphs due to Groshaus and Szwarcfiter~\cite{MR2383735}.
A graph is said to be \emph{closed neighbourhood-Helly} if the family of closed neighbourhoods of vertices in $G$ have the Helly property, that is, every pairwise intersecting subfamily of the closed neighborhoods of $G$ has a common element. 
A graph $G$ is said to be \emph{hereditary closed neighbourhood-Helly} if every induced subgraph of $G$ is closed neighbourhood-Helly.
Recall that the \emph{net} is the graph $S_{1,1,1}$, that is, the graph obtained from the complete graph $K_3$ by attaching a pendant edge to each vertex.
The \emph{co-net} is the complement of the net.
(The co-net is referred to in the literature under various other names, including the \emph{$3$-sun} and the \emph{Hajós graph}.)

\begin{theorem}[Groshaus and Szwarcfiter~\cite[Theorem 3.2]{MR2383735}]\label{thm:GH:hcnH}
A graph $G$ is hereditary closed neighbourhood-Helly if and only if it contains neither $C_4$, $C_5$, $C_6$ nor the co-net as an induced subgraph.
\end{theorem}

\getkeytheorem{netfreecochordal}

\begin{proof}
Let $H$ be a co-net-free chordal graph and let $G$ be the complement of $H$.
If~$H$ has diameter at least~$3$, then the equivalence holds by \zcref{characterization-co-chordal}, since each of the graphs $P_4$, $K_{2,3}$, $K_{1,\ell}$, $\doublestar{k}{\ell}$, $\turan{2\ell}{\ell}$, and $\turan{2\ell-1}{\ell}$ for ${k \geq 1}$ and~${\ell \geq 2}$ is net-free.

Suppose now that~$H$ has diameter at most~$2$.
Then, every pair of closed neighbourhoods has a non-empty intersection.
Since $H$ contains neither $C_4$, $C_5$, $C_6$ nor the co-net as an induced subgraph, by~\zcref{thm:GH:hcnH}, the closed neighbourhoods of $H$ have a common element, that is, $H$ has a universal vertex.
Hence, $G$ has an isolated vertex and neither of the two conditions in the theorem holds.
\end{proof}

In particular, the generalized Kriesell's conjecture holds for net-free co-chordal graphs.

Let us also note that, since every interval graph is strongly chordal, and every strongly chordal graph is a co-net-free chordal graph, \zcref{characterization-net-free-co-chordal} implies analogous characterizations of non-trivially minimally tough graphs within classes of complements of interval graphs and complements of strongly chordal graphs. 
The statements are the same as in \zcref{characterization-net-free-co-chordal}, with  ``net-free co-chordal graph'' replaced with ``complement of an interval graph'' and ``complement of a strongly chordal graph'', respectively.

\subsection{Complements of forests}
\label{sec:coforests}

Lastly, we derive \zcref{characterization-co-forests} from \zcref{characterization-co-chordal}.

\getkeytheorem{coforests}

\begin{proof}
Let $F$ be a forest and let $G$ be the complement of $F$.
If~$F$ has diameter at most~$1$, then $G$ is edgeless and neither of the two conditions in the theorem holds.
If~$F$ has diameter $2$, then $F$ is isomorphic to a star with at least two leaves, in which case $G$ is disconnected, and again neither of the two conditions in the theorem holds.

Assume now that $G$ has co-diameter at least~$3$.
Since $G$ is co-chordal, \zcref{characterization-co-chordal} implies that $G$ is non-trivially minimally tough if and only if it is isomorphic to one of $P_4$, $K_{2,3}$, $K_{1,\ell}$, $\doublestar{k}{\ell}$, $\turan{2\ell}{\ell}$, or $\turan{2\ell-1}{\ell}$ for some ${k \geq 1}$ and~${\ell \geq 2}$. 
Observe that $\turan{3}{2}\cong K_{1,2}$.
Therefore, since neither of the graphs $K_{2,3}$, $K_{1,\ell}$ for $\ell\ge 3$, and $\doublestar{k}{\ell}$ with ${k \geq 1}$ and~${\ell \geq 2}$ is a complement of a forest, the theorem follows. 
\end{proof}

As a corollary, similar to \zcref{corollary:GKC2}, we obtain that the generalized Kriesell's conjecture holds for complements of forests.

\section{Open problems}
\label{sec:conclusion}

The main problem left open by this work is whether \zcref{P4free:mintough,characterization-co-chordal,characterization-net-free-co-chordal,characterization-co-forests,complete-char-multipar} admit a common generalization, provided by an answer to the following.

\begin{problem}
Characterize minimally tough weakly chordal graphs.
\end{problem}

This problem includes, as a special case, a characterization of non-trivially minimally tough chordal graphs. 
They were conjectured not to exist for toughness exceeding $1$ by Dallard et al.~\cite{Dall23}.

Another interesting special case is given by the case of co-chordal graphs.
Those that are also chordal are split, for which minimal toughness has been completely characterized~\cite{Kato23}. 
For the general case of co-chordal graphs, we obtained a partial characterization (see \zcref{characterization-co-chordal,characterization-net-free-co-chordal}), leaving open only the case when the complement has diameter $2$ and contains an induced co-net. 
For this case, we conjecture that the non-trivially minimally tough graphs are exactly the graphs obtained from the disjoint union of three stars~$K_{1,\ell}$ for a positive integer~$\ell$ by adding a triangle between their centers; we denote these graphs by~$\triplestar{\ell}{\ell}{\ell}$. 
These graphs are the only minimally tough split graphs with co-diameter~$2$; see Katona and Varga~\cite{Kato23}. 

\begin{conjecture}\label{conj:cochordal-diam-2}
Let $G$ be a co-chordal graph with co-diameter~$2$. 
Then $G$ is non-trivially minimally tough if and only if $G$ is isomorphic to~$\triplestar{\ell}{\ell}{\ell}$ for some~${\ell \geq 1}$.
\end{conjecture}

Since all the graphs $\triplestar{\ell}{\ell}{\ell}$ have toughness at most $1/2$, a possibly easier first step towards proving \zcref{conj:cochordal-diam-2} could be to address the following. 

\begin{question}\label{question:cochordal-diam-2}
    What is the maximum toughness of a minimally tough co-chordal graph with co-diameter~$2$?
\end{question}
We have shown that the generalized Kriesell's conjecture holds for co-chordal graphs with co-diameter at least $3$ (see \zcref{corollary:GKC2}). Notice that \zcref{conj:cochordal-diam-2} would imply that the generalized Kriesell's conjecture holds for the entire class of co-chordal graphs.

\section*{Acknowledgements}

This work is supported in part by the Slovenian Research and Innovation Agency (I0-0035, research program P1-0285 and research projects J1-60012, J1-70035, J1-70046, and N1-0370) and by the research program CogniCom (0013103) at the University of Primorska.

\printbibliography

\end{document}